\pdfoutput=1
\RequirePackage{ifpdf}
\ifpdf 
\documentclass[pdftex]{sigma}
\else
\documentclass{sigma}
\fi

\usepackage[all]{xy}
\usepackage[mathscr]{euscript}

\def\b{\bullet}

\newcommand{\bs}[1]{\boldsymbol{#1}}
\newcommand{\dg}[1]{{}^{\ast\text{-gr}}#1}

\newcommand{\ol}[1]{\overline{#1}}
\def\Id{\operatorname{I}}

\DeclareMathOperator{\Ker}{Ker}
\newcommand{\h}{\mathrm{H}}

\DeclareMathOperator{\coker}{Coker}

\DeclareMathOperator{\im}{Im}
\DeclareMathOperator{\K}{K}
\DeclareMathOperator{\Tor}{Tor}
\DeclareMathOperator{\I}{Id}

\newcommand\ot{\otimes}
\newcommand\rar{\rightarrow}
\newcommand\op{\oplus}
\newcommand\kal{\mathscr}
\newcommand\bb{\mathbb}

\newcommand{\Ap}[1]{A_+^{(#1)}}
\newcommand{\Cp}[1]{C_+^{(#1)}}
\newcommand{\gen}[1]{\langle #1 \rangle}
\newcommand{\br}[1]{\{ #1 \}}

\newcommand{\Coker}{\operatorname{Coker}}

\newcommand{\kap}{\Bbbk^a[\kal{P}]}

\def\m{\mu}
\newcommand{\Od}[2]{\Omega_{#1}(#2)}
\newcommand{\Ou}[2]{\Omega^{#1}\big(#2\big)}
\newcommand{\Pa}[2]{\mathscr{P}_{#2}(#1)}

\newcommand\xrar{\xrightarrow}

\newcommand{\N}{{\mathbb N}}

\newcommand{\Ext}{\operatorname{Ext}}
\newcommand\Hom{\operatorname{Hom}}

\newcommand{\pa}{\partial}
\newcommand{\al}{\alpha}
\newcommand{\be}{\beta}

\newcommand{\de}{\delta}

\numberwithin{equation}{section}

\newtheorem{Theorem}{Theorem}[section]
\newtheorem{Corollary}[Theorem]{Corollary}
\newtheorem{Lemma}[Theorem]{Lemma}
\newtheorem{Proposition}[Theorem]{Proposition}
{ \theoremstyle{definition}
\newtheorem{Definition}[Theorem]{Definition}
\newtheorem{Remark}[Theorem]{Remark}
}

\begin{document}


\newcommand{\arXivNumber}{1504.03548}

\renewcommand{\PaperNumber}{092}

\FirstPageHeading

\ShortArticleName{Further Properties and Applications of Koszul Pairs}

\ArticleName{Further Properties and Applications of Koszul Pairs}

\Author{Adrian MANEA and Drago\c{s} \c{S}TEFAN}

\AuthorNameForHeading{A.~Manea and D.~\c{S}tefan}

\Address{Faculty of Mathematics and Computer Science, University of Bucharest,\\
 14 Academiei Str., Bucharest Ro-010014, Romania}
\Email{\href{mailto:adrian.manea@fmi.unibuc.ro}{adrian.manea@fmi.unibuc.ro}, \href{mailto:dragos.stefan@fmi.unibuc.ro}{dragos.stefan@fmi.unibuc.ro}}

\ArticleDates{Received May 19, 2016, in f\/inal form September 08, 2016; Published online September 14, 2016}

\Abstract{Koszul \looseness=-1 pairs were introduced in [arXiv:1011.4243] as an instrument for the study of Koszul rings. In this paper, we continue the enquiry of such pairs, focusing on the description of the second component, as a follow-up of the study in [arXiv:1605.05458]. As such, we introduce Koszul corings and prove several equivalent characterizations for them. As applications, in the case of locally f\/inite $R$-rings, we show that a graded $R$-ring is Koszul if and only if its left (or right) graded dual coring is Koszul. Finally, for f\/inite graded posets, we obtain that the respective incidence ring is Koszul if and only if the incidence coring is so.}


\Classification{16E40; 16T10; 16T15}

\section{Introduction}
The classical approach on Koszul rings is due to~\cite{bgs}, which def\/ines an $\bb{N}$-graded ring $A = \op_{n \in \bb{N}} A^n$ to be Koszul if and only if $A^0$ is a semisimple ring and there exists a resolution~$P_\bullet$ of $A^0$ by projective graded left $A$-modules such that each term $P_n$ is generated by its homogeneous component of degree~$n$. Koszul rings evolved as natural generalisations of Koszul algebras, which in turn were discovered by Priddy in~\cite{Pr}. Since their early developments, Koszul algebras and rings proved to be very useful tools in various f\/ields of mathematics, as are representation theory, algebraic geometry, algebraic topology, quantum groups, combinatorics and many more. A comprehensive read is, for example, \cite{PP} and the references therein.

While studying certain cohomological properties of Koszul rings in~\cite{jps}, the authors were led to a new notion, the so-called \emph{Koszul pairs}, which we explain in brief. Let $R$ be a semisimple ring. A graded $R$-ring is a graded algebra in the tensor category of $R$-bimodules with respect to the tensor product (of bimodules). A graded $R$-ring $A = \op_{n \in \bb{N}} A^n$ is called \emph{connected} if $A^0 = R$. Connected graded $R$-corings are def\/ined by duality. By def\/inition, an \emph{almost Koszul pair} consists of a graded connected $R$-ring $A$ and a graded connected $R$-coring $C$, together with an $R$-bimodule isomorphism $\theta_{A,C}\colon C_1 \to A^1$. These data must be compatible, in the sense that the composition of the three maps below must be zero:
\begin{gather}\label{ec:theta}
C_2 \xrar{\Delta} C_1 \ot C_1 \xrar{\theta_{C,A} \ot \theta_{C,A}} A^1 \ot A^1 \xrar{\mu^{1,1}} A^2,	
\end{gather}
where $\Delta_{1,1}$ and $\mu^{1,1}$ denote the components of (co)multiplication maps for $C$ and $A$, respectively.

By \cite{jps}, to an almost-Koszul pair correspond three chain complexes and three cochain complexes, which measure how far is $A$ from being a Koszul ring. More precisely, an important feature of an almost Koszul pair is that any of the corresponding six complexes is exact if and only if all of them are so. In this case, the pair $(A,C)$ is called \emph{Koszul}. Furthermore, one proves that a connected graded $R$-ring $A$ is Koszul if and only if there exists a connected graded $R$-coring $C$ such that $(A,C)$ is a Koszul pair.

In this paper, we continue the work on Koszul pairs and their applications, having in view two aims. Firstly, following our results in \cite{mst} which include several characterizations Koszul rings, we prove that after properly def\/ining the dual notion, that of \emph{Koszul corings}, an appropriate characterization theorem holds true. Furthermore, we obtain a new proof of the fact that a~Koszul $R$-ring is quadratic and also a corresponding result for corings. In this sense, we refer to Theorem~\ref{thm:kcoring} for the general properties of Koszul corings and to Corollary~\ref{co:cquad} for quadraticity.

As applications, we remark that in the locally f\/inite case, as one could expect, a ring is Koszul if and only if its graded dual coring is so, cf.\ Theorem~\ref{te:kdual}. In the proof of this result, we use again some Koszul pairs involving the graded duals as the starting structures. In particular, the incidence ring of a graded poset is Koszul if and only if its incidence coring is Koszul as well, as per Theorem~\ref{te:incidence-path}. In this class of f\/inite graded posets, we have obtained in~\cite{mst} some particular examples of Koszul posets and in this article we show that they also provide examples of Koszul corings. Taking this into account, this paper can be seen as an announced follow-up of~\cite{mst}.

More examples of Koszul corings, related to Hopf algebras, will be considered in a forthcoming paper.

Since the literature is abundant on this subject, some remarks regarding the scope and position of the present paper are due. First of\/f, a vast majority of the theory is known and developed for the algebra or $R$-rings case. In this respect, we introduce and study the dual case, that of corings. By these means, we obtain new insights on the theory. Moreover, although recent developments of the theory have dealt with a categorical setting (e.g., \cite{ke} or~\cite{mos}), the present study is, in majority, self-contained, by using the tool of Koszul pairs that the second author most introduced in~\cite{jps}.

We acknowledge as well that further generalizations on the theory of Koszul rings were treated, for example, in~\cite{grs,md1,md2}, eliminating the semisimplicity condition on the part of degree zero. Many other directions which relax the starting basic assumptions were considered in the literature (as it is the case, for example, in~\cite{ber} and some other articles of the same author). However, as remarked in the previous paragraph, the present article focuses on the dualization of the more ``classical'' case by means of Koszul pairs. This way, the overlap or even connection with the references above is at its minimum and the subjects treated therein exceed the purposes of the present paper.

The article is organized as follows. In Section~\ref{section2}, we recall the preliminary notions and results regarding Koszul pairs that are needed in the paper. Section~\ref{section3} brief\/ly revisits the case of Koszul $R$-rings, whereas Section~\ref{section4} introduces and discusses Koszul $R$-corings. Sections~\ref{section5} and~\ref{section6} study applications and examples.

\section{Preliminaries}\label{section2}
In this section we recall some basic concepts and notations from \cite{jps} and then we shall prove some new preliminary results, which are needed later on.
	
\subsection{Connected (co)rings} \label{corings} Let $R$ be a semisimple ring that we f\/ix throughout the article. Since we will always work with algebras and coalgebras in the tensor category of $R$-bimodules, an unadorned tensor product $\ot$ will mean $\ot_R$. Let $V$ be an $R$-bimodule. The notation $V^{(n)}$ will be used for the $n$-th tensor power $V \ot \cdots \ot V$. Conventionally, $V^{(0)} = R$. Similarly, for any bimodule morphism $f\colon M \to N$, the tensor product $f \ot \cdots \ot f$ with $n$ factors will be denoted by $f^{(n)}$. For a set $X$, the identity morphism will be denoted by $\operatorname{Id}_X$ or simpler, $\operatorname{I}_X$ or even $X$ when there is no risk of confusion.

An $\bb{N}$-graded algebra $A = \op_{n \in \bb{N}} A^n$ is called an \emph{$R$-ring} if it is an algebra in the tensor category of $R$-bimodules and it is connected if $A^0=R$. Dually, a connected $R$-coring is a graded coalgebra $C = \op_{n \in \bb{N}}$ in the category of $R$-bimodules such that $C_0=R$. The augmentation ideal of $A$ will be denoted by $A_+ = \op_{n \in \bb{N}^\ast} A^n$ and similarly for~$C$. Note that the comultiplication map $\Delta$ induces a canonical coassociative map $\Delta_+ \colon C_+ \to C_+ \ot C_+$, where $C_+$ can also be viewed as~$C/R$.

For a graded connected $R$-ring $A$, the multiplication $\mu$ is def\/ined by some maps $\mu^{p,q} \colon$ \mbox{$A^p \ot A^q \to A^{p+q}$}. When all of these maps are surjective, for all $p,q \geq 0$, we say that $A$ is \emph{strongly graded}. Equivalently, $A$ is strongly graded if and only if the iterated multiplication $\mu_n \colon (A^1)^{(n)} \to A^n$ is surjective for all $n \geq 2$.

Similarly, for a graded connected $R$-coring $C$, the comultiplication is def\/ined by some $R$-bimodule maps $\Delta_{p,q} \colon C_{p+q} \to C_p \ot C_q$ and the coring $C$ is called \emph{strongly graded} if and only if all of them are injective.

Let $\Delta^1 = \operatorname{Id}_{C_1}$ and for $n \geq 2$ def\/ine a map $\Delta^n\colon C_n \to C_1^{(n)}$ by the recurrence relation
\begin{gather*}
	\Delta^n = \big(\operatorname{Id}_{C_1} \ot \Delta^{n-1}\big) \circ \Delta_{1,n-1}.
\end{gather*}

It is immediate from the def\/inition and the coassociativity property that the following equality holds true
\begin{gather*}
	\Delta^{p+q} = (\Delta^p \ot \Delta^q) \circ \Delta_{p,q}.
\end{gather*}
Another remark is that the strong grading on $C$ can be seen to be equivalent to the injectivity of all of the maps $\Delta^n$, $\forall\, n$ and further that $\Delta_{1,n}$ is injective for all $n$ if and only if $\Delta_{n,1}$ is injective for all~$n$.

When $C$ is a connected $R$-coring, the unit of $R$ is a group-like element, that is $\Delta(1) = 1 \ot 1$, cf.~\cite{jps}. Therefore, we can speak of primitive elements of~$C$, namely those $c \in C$ for which $\Delta(c) = 1 \ot c + c \ot 1$. The set of all primitive elements in $C$ contains $C_1$ and will be denoted by~$PC$. In general, the inclusion of $C_1$ in $PC$ is strict, but as per \cite[Lemma~1.4]{mst}, $PC = C_1$ if and only if~$C$ is strongly graded.

Indecomposable elements of an augmented $R$-ring correspond by duality to primitive elements. They were introduced by May in~\cite{may} and we will see how they relate to the strong grading on an $R$-ring. In the graded case, the $R$-bimodule $QA$ of indecomposable elements is def\/ined by the exact sequence
\begin{gather*} A_+ \ot A_+ \xrar{\mu} A_+ \to QA \to 0.\end{gather*}
There is a canonical morphism $A^1 \to QA$, which maps $a \in A^1$ to its class $a + A_+^2 \in QA$. This map has a left inverse, induced by the projection $A_+ \to A^1$.

There is a canonical morphism $A^1 \to QA$ which maps an element in $A^1$ to its class modulo the square of the augmentation ideal $A_+$ in~$QA$. This map has a left inverse, induced by the projection $A_+ \to A^1$.

We can prove a f\/irst result, which is a dual version of \cite[Lemma~1.4]{mst} and which shows the role of $QA$, the bimodule of indecomposable elements in an $R$-ring. Moreover, the following lemma will be used for obtaining our main characterization theorem for Koszul corings (i.e., Theorem~\ref{thm:kcoring}).

\begin{Lemma} \label{lema:ring}
Let $A$ be a graded and connected $R$-ring.
\begin{enumerate}\itemsep=0pt
\item[$1.$] $A$ is strongly graded if and only if the canonical map $QA\to A^1$ is injective, if and only if the canonical map $A^1\to QA$ is surjective.

\item[$2.$] If $A$ is strongly graded, $B$ is a connected graded $R$-ring and $g\colon B \rar A$ is a morphism of graded $R$-rings such that its components $g^0$ and $g^1$ are surjective, then $g$ is surjective.

\item[$3.$] Let $A=\oplus_{n,m\in\bb{N}}A^{n,m}$ be bigraded. If $\operatorname{gr}A$ is strongly graded and $A^{n,m}=0$ for $n=0,1$ and all $m \neq n$, then $A^{n,m}=0$ for all $n\geq 2$, $m\neq n$.
\end{enumerate}
\end{Lemma}

\begin{proof}
Since $QA\to A^1$ is a left inverse of $A^1\to QA$ it follows that the former map is injective if and only if the latter is surjective. On the other hand, $A$ is strongly graded if and only if $A_n\subseteq A_+^2$, for all $n\geq 2$. Consequently, the f\/irst statement holds, as the above inclusions are true if and only if the map $A^1\to QA$ is surjective.

The proof of the second statement follows as in the case of corings \cite[Lemma~1.4]{mst}, using the diagram below
\begin{gather*}
 \xymatrixcolsep{2pc}\xymatrix{\ar[r]^-{\mu^{n,1}_{B}} B_{n} \otimes B_{1} \ar[d]_-{g_n \otimes g_1} &B_{n+1} \ar[d]^-{g_{n+1}} \\
A_{n} \otimes A_{1} \ar[r]^-{\mu^{n,1}_{A}} & A_{n+1}. }
\end{gather*}
For proving the last part of the lemma we def\/ine $\operatorname{Diag}(A)$ to be graded subring of $\operatorname{gr}(A)$ given by $\operatorname{Diag}(A)=\oplus_{n\in\N} A^{n,n}$. Let $\iota\colon \operatorname{Diag}(A)\to \operatorname{gr}(A)$ denote the inclusion map. Obviously, by the standing assumptions, $\iota^0$ and $\iota^1$ are surjective maps. Hence, by the second part of the lemma, $\iota$~is surjective. Thus $A^{n,m}$ must be zero for all $n$ and $m$, with $n\neq m$.
\end{proof}

\subsection[The $R$-ring $\langle V,W\rangle$ and the $R$-coring $\{V,W\}$]{The $\boldsymbol{R}$-ring $\boldsymbol{\langle V,W\rangle}$ and the $\boldsymbol{R}$-coring $\boldsymbol{\{V,W\}}$}\label{shriek}

For an $R$-bimodule $V$ and a sub-bimodule $W\subseteq V\ot V$, we def\/ine the $R$-ring $\langle V, W \rangle$ to be the quotient $T^a_R (V)/\langle W \rangle$ of the tensor algebra of the $R$-bimodule $V$ by the two-sided ideal generated by $W$. Note that $\langle W\rangle=\sum\limits_{n\in\mathbb{N}} \langle W\rangle^n$, where
 \begin{gather*}
 \langle W\rangle^n=\sum_{i=1}^{n-1}V^{(i-1)}\ot W\ot V^{ (n-i-1)}. \end{gather*}
For $V$ and $W$ as above, one also constructs a graded $R$-coring $\br{V,W}$ by taking $\br{V,W}_0=R$ and $\br{V,W}_1=V$. For all $n \geq 2$ def\/ine
\begin{gather*}
\br{V,W}_n=\bigcap_{p=1}^{n-1} V^{(p-1)} \otimes W \otimes V^{ (n-p-1)}.
\end{gather*}
As it is shown in~\cite{jps}, the direct sum $\br{V,W}=\oplus_{n \in \mathbb{N}} \br{V,W}_n$ is a graded subcoring of $T^c_R(V)$, the tensor coalgebra of the $R$-bimodule~$V$.

If $A$ is a connected graded $R$-ring and $C$ is a connected graded $R$-coring then the graded coring $\{A^1,\Ker\mu^{1,1}\}$ and the graded ring $\gen{C_1,\im\Delta_{1,1}}$ will be denoted by $A^!$ and $C^!$, respectively and called the \emph{shriek coring} and the \emph{shriek ring}, respectively. The homogeneous component of degree~$n$ of~$A^!$ will be denoted by $A^!_n$. To simplify the notation, for the ring $C^!$ we shall write $C^!_n$ instead of~$(C^!)^n$.

\begin{Remark} In the literature, the shriek construction (also known as \emph{the quadratic dual} \cite[Section~2.8]{bgs}, for example) is def\/ined and used slightly dif\/ferent that our setting. However, we note that throughout this paper, the only meaning of the shriek structures is that introduced above. Note that, unlike the classical case of~\cite{bgs}, for example, here the shriek construction changes an $R$-ring into an $R$-coring and viceversa. \end{Remark}

\subsection{Bigraded corings}
Some basic facts regarding bigraded corings are due, since many of the structures which we will use are of this kind.

An $R$-coring $C$ is called \emph{bigraded} if it has a decomposition as a direct sum of $R$-bimodules $C=\oplus_{n,m \in \mathbb{N}} C_{n,m}$ such that its comultiplication induces a collection of maps
\begin{gather*} C_{n+n',m+m'} \xrightarrow{\Delta_{n,m}^{n',m'}} C_{n,m} \otimes C_{n',m'}.\end{gather*}
For this case, coassociativity means that the diagram below is commutative, for all positive integers $m$, $n$, $p$, $m'$, $n'$ and $p'$,
\begin{gather*}
 \xymatrixcolsep{5pc}\xymatrix{ C_{n+m+p,n'+m'+p'} \ar[d]_-{\Delta_{n,n'}^{m+p,m'+p'}} \ar[r]^-{\Delta_{n+m,n'+m'}^{p,p'}} & C_{n+m,n'+m'} \otimes C_{p,p'} \ar[d]^-{\Delta_{n,n'}^{m,m'} \otimes \Id_{C_{p,p'}}} \\
C_{n,n'} \otimes C_{m+p,m'+p'} \ar[r]_-{\Id_{C_{n,n'}} \otimes \Delta_{m,m'}^{p,p'}} & C_{n,n'} \otimes C_{m,m'} \otimes C_{p,p'}.}
\end{gather*}
By def\/inition, we also impose to the counit to vanish on $C_{n,m}$, provided that either $n>0$ or $m>0$.

Starting with a bigraded coring, one can associate to it a graded coring $\operatorname{gr}(C)$, whose homogeneous component of degree $n$ is $\operatorname{gr}_n(C)=\oplus_m C_{n,m}$. Therefore, def\/ine $\operatorname{gr}(C):=\oplus_n \operatorname{gr}_n(C)$. We will be interested only in \textit{connected bigraded corings}, namely those bigraded corings for which $C_{0,0}=R$ and $C_{0,m}=0$, for $m>0$. Note that $C$ to be connected implies $\operatorname{gr}(C)$ is connected as well.

Let $C$ be a (connected) bigraded coring. Def\/ine
\begin{gather*} C'_{n,m}=\begin{cases} C_{n,m}, & n=m, \\ 0, & n\neq m.\end{cases}\end{gather*}
Then ${C'}:=\oplus_{n,m} {C}_{n,m}'$ becomes a (connected) bigraded coring. We denote the graded coring $\operatorname{gr}({C'})$ associated to it by $\operatorname{Diag}(C)$.

Keeping the notations and the context above, there exist canonical $R$-bimodule morphisms $\pi_{n,m} \colon C_{n,m} \to C'_{n,m}$ which are identities on $C_{n,n}$ and zero maps in rest. Their collection, $\pi = \{\pi_{n,m}\}_{n,m}$, def\/ines a morphism in the category of graded corings. Since $\operatorname{Diag}(C) = \operatorname{gr}(C')$, it follows that the map $\pi$ induces a morphism at the level of graded corings, namely $\operatorname{gr}(\pi) \colon \operatorname{gr}(C) \to \operatorname{Diag}(C)$. Surely the homogeneous component of degree $n$ of $\operatorname{Ker}\pi$ is $\op_{n \neq m} C_{n,m}$.

Note that one can def\/ine by duality the corresponding notions of bigraded $R$-rings. We omit the details here, since this case is better known than the one for corings. See, for example, \cite{pi} for the dif\/ferential graded setting.

\subsection{The normalized (co)chain complex} \label{sec-t(a)}
We shall compute $T_n(A):=\Tor_n^A(R,R)$ as the $n$th homology group of the normalized bar complex $(\Omega_\bullet(A), d_\bullet)$, where $\Omega_n(A)=\Ap{n}$. The morphisms $d_n\colon \Omega_n(A)\to\Omega_{n-1}(A)$ are def\/ined by $d_1=0$ and, for $n>1$,
\begin{gather*}
d_n(a_1\ot\cdots \ot a_n)=\sum_{i=1}^{n-1}(-1)^{i-1}a_1\ot\cdots \ot a_ia_{i+1}\ot\cdots\ot a_n.
\end{gather*}
Since the normalized bar complex has a canonical structure of DG-coalgebra in the category of $R$-bimodules with respect to the comultiplication $\Delta_{p,q}(a_1\ot\cdots \ot a_{p+q})=(a_1\ot\cdots \ot a_{p})\ot(a_{p+1}\ot\cdots \ot a_{p+q})$ it follows that $T(A)=\oplus_{n \in \mathbb{N}} T_n(A)$ has a canonical structure of connected $R$-coring.
	
One can express the complex $\Omega_\bullet(A)$ as a direct sum $\Omega_\bullet(A)=\oplus_{m \geq 0} \Omega_\bullet(A,m)$ of subcomplexes. As in \cite[Section~1.6]{mst}, we introduce the following terminology and notation. An $n$-tuple $\boldsymbol{m}=(m_1,\dots,m_n)$ is called \textit{a positive $n$-partition of $m$} if and only if $\sum\limits_{i=1}^n m_i=m$ and all $m_i$ are positive. The set of all positive $n$-partitions of $m$ will be denoted by $\mathscr{P}_n(m)$. Furthermore, if $\boldsymbol{m}=(m_1,\dots,m_n)$ is a positive $n$-partition and $A$ is a connected $R$-ring then for the tensor product $A^{m_1}\ot\cdots\ot A^{m_n}$ we shall use the notation $A^{\boldsymbol{m}}$. For a positive $n$-partition $\bs{m}=(m_1,\dots,m_n)$ of~$m$, the multiplication $\m$ of $A$ induces bimodule maps $\mu^{\bs{m}}\colon A^{\bs{m}}\to A^m$ and $\mu_{\bs{m}}\colon (A^1)^{(m_1)}\ot\cdots \ot(A^1)^{(m_n)}\to A^{\bs{m}}$. Note that, by def\/inition, $\mu_{\bs{m}}=\m(m_1)\ot\cdots\ot\m(m_n)$.

Hence, with this notation, $\Omega_\bullet(A,m)$ is the following subcomplex of $\Omega_\bullet(A)$:
\begin{gather*}
0 \xleftarrow{d_0^m} 0 \xleftarrow{d_1^m} \bigoplus \limits_{\boldsymbol{m}_1\in\mathscr{P}_1(m)} A^{\boldsymbol{m}_1} \xleftarrow{d_2^m} \bigoplus \limits_{\boldsymbol{m}_2\in\mathscr{P}_2(m)} A^{\boldsymbol{m}_2}\xleftarrow{d_3^m} \cdots \xleftarrow{d^m_n}  \bigoplus \limits_{\boldsymbol{m}_n\in\mathscr{P}_{n}(m)} A^{\boldsymbol{m}_n}\xleftarrow{\ \ \ }\cdots .
\end{gather*}
Of course, there is only one $1$-partition of $m$, namely $\boldsymbol{m}_1=(m)$. Thus the f\/irst direct sum in~$\Omega_\bullet(A,m)$ coincides with $A^m$. The homology in degree $n$ of $\Omega_\bullet(A,m)$ will be denoted either by~$T_{n,m}(A)$ or~$\operatorname{Tor}_{n,m}^A(R,R)$. Clearly, we have $T(A)=\oplus_{m\geq 0} T_{n,m}(A)$ and this decomposition is compatible with the coring structure, in the sense that $T(A)$ is a bigraded coring with~$T_{n,m}(A)$ as $(n,m)$-homogeneous component.
	
Dually, for a connected $R$-coring $C$, the normalized bar cochain complex $(\Ou{\bullet}{C}, d^\b)$ is def\/ined by $\Ou{n}{C}=\Cp{n}$ and $d^0=0$, while
\begin{gather*}
d^n=\sum_{i=1}^{n}(-1)^{i-1}\I_{C_+^{(i-1)}}\ot\;\Delta_+\ot\I_{C_+^{(n-i)}}.
\end{gather*}
Here $C_+:=C/C_0$ and $\Delta_+\colon C_+\to C_+\ot C_+$ is the map induced by the comultiplication of~$C$. It is well-known that $E^n(C)=\operatorname{Ext}^n_C(R,R)$ is the $n$th cohomology group of $\Omega^\bullet(C)$ and that $E(C)=\oplus_{n \in \mathbb{N}} \ E^n(C)$ is a connected $R$-ring with respect to the multiplication induced by the DG-algebra structure (in the category of $R$-bimodules) on $\Ou{\b}{C}$ which is def\/ined by concatenation of tensor monomials (see, e.g., \cite[Section~1.1]{PP}).

The complex $\Omega^\b(C)$ is a direct sum of subcomplexes $\Omega^\b(C,m)$, that are def\/ined as follows. For a positive $n$-partition of $m$ let $C_{\boldsymbol{m}}:=C_{m_1}\ot \cdots\ot C_{m_n}$. Hence $\Omega^\b(C,m)$ is the subcomplex
\begin{gather*}
\ 0 \xrightarrow{\ \ \ } 0 \xrightarrow{d_m^0} \bigoplus \limits_{\boldsymbol{m}_1\in\mathscr{P}_1(m)} C_{\boldsymbol{m}_1} \xrightarrow{d_m^1} \bigoplus \limits_{\boldsymbol{m}_2\in\mathscr{P}_2(m)} C_{\boldsymbol{m}_2} \xrightarrow{d_m^2} \cdots \xrightarrow{d^{n-1}_m}  \bigoplus \limits_{\boldsymbol{m}_n\in\mathscr{P}_{n}(m)} C_{\boldsymbol{m}_n} \xrightarrow{d_m^n}\cdots .
\end{gather*}
Of course, $\Omega^1(C,m)=C_m$. The homology in degree $n$ of $\Omega^\bullet(C,m)$ will be denoted either by $E^{n,m}(C)$ or $\operatorname{Ext}^{n,m}_C(R,R)$. Clearly, we have $E(C)=\oplus_{m\geq 0}E^{n,m}(C)$ and this decomposition is compatible with the ring structure in the sense that $E(C)$ is a bigraded ring with $E^{n,m}(C)$ as $(n,m)$-homogeneous component. For more details the reader is referred to \cite[Section~1.15]{jps}.

\subsection{Almost-Koszul pairs} \label{sec:almost}
As mentioned in the introduction, an almost Koszul pair $(A,C)$ consists of a connected and graded $R$-ring $A$ and a connected graded $R$-coring $C$, together with an isomorphism of $R$-bimodules $\theta_{A,C} \colon C_1 \rightarrow A^1$ which satisf\/ies the relation \eqref{ec:theta}. Using the graded version of Sweedler's notation for the comultiplication of $C$, the above equation is equivalent to
\begin{gather} \label{ec:qksz}
	\sum \theta_{C,A}(c_{1,1})\theta_{C,A}(c_{2,1})=0,
\end{gather}
where $c$ is an arbitrary element of $C_2$.

If $(A,C)$ and $(B,D)$ are almost-Koszul pairs, then a morphism of almost-Koszul pairs is a~couple $(\phi,\psi)$, where $\phi\colon A \rar B$ is a morphism of graded $R$-rings and $\psi\colon C \rar D$ is a morphism of graded $R$-corings such that they commute with the isomorphisms $\theta_{A,C}$ and $\theta_{B,D}$. Henceforth, the diagram below is commutative,
\begin{gather*}
 \xymatrix{ A^1 \ar[r]^-{\phi_1} \ar[d]_-{\theta_{A,C}} & B^{1} \ar[d]^-{\theta_{B,D}} \\
C_1 \ar[r]^-{\psi_1} & D_1.}
\end{gather*}
A f\/irst example of an almost-Koszul pair is given in \cite[Proposition 1.8]{jps}. For any connected and strongly graded $R$-ring $A$, the pair $\big(A,T(A)\big)$ is almost-Koszul. Here, the coring structure of $T(A)$ is def\/ined as in Section~\ref{sec-t(a)}, and the map $\theta_{A,T(A)}\colon T_1(A)\to A^1$ is induced by the projection $A_+\to A^1$. Note that $T_1(A)=A_+/A_+^2\cong A^1$, as $A$ is strongly graded, so $\theta_{A,T(A)}$ is an isomorphism.

The couple $(A,A^!)$ is another example of an almost-Koszul pair. Recall that $A^!_1=A^1$, so we can take $\theta_{A,A^!}:=\Id_{A^1}$. The condition~\eqref{ec:theta} is trivial in this particular case as, by construction, $A^!_2=\Ker\mu^{1,1}$.

Dually, if $C$ is a connected and strongly graded $R$-coring, the pair $(E(C),C)$ is almost-Koszul, cf.~\cite[Proposition~1.18]{jps}. In this example the $R$-ring structure of $E(C)$ is def\/ined in Section~\ref{sec-t(a)}. Since $E^1(C)=\Ker\Delta_+$, the map $\theta_{E(C),C}\colon C_1\to E^1(C)$, given by $\theta_{E(C),C}(c):=c+C_0\in C/C_0$, is well def\/ined and it is an isomorphism of $R$-bimodules.

The couple $(C^!,C)$ can be seen as an almost-Koszul pair with respect to the map $\theta_{C^!,C}=\Id_{ C_1}$. The relation \eqref{ec:theta} is verif\/ied since $C^!_2:=(C_1\ot C_1)/\im\Delta_{1,1}$ and the component $C_1^!\ot C_1^!\to C^!_2$ of the multiplication on $C^!$ coincides with the canonical projection from $C_1\ot C_1$ onto $C^!_2$.

\subsection{Koszul pairs} \label{fa:Koszul pairs}
Following \cite{jps} we shall brief\/ly recall the def\/inition of Koszul pairs, which are our main tool to investigate Koszul $R$-(co)rings. For any almost-Koszul pair $(A,C)$ and $n\geq 0$ we def\/ine a~complex of graded right $C$-comodules by
\begin{gather*} \operatorname{K}_r^{-1}(A,C)=R \qquad \text{and} \qquad \operatorname{K}^n_r(A,C)=A^n \otimes C, \qquad \forall\, n \geq 0.\end{gather*}
The dif\/ferential map $d_{r}^{n}\colon A^{n}\otimes C\rightarrow A^{n+1}\otimes C$ is zero on $A^{n}\otimes C_{0}$ and, for $p>0$ and $a\otimes c\in A^{n}\otimes C_{p},$
\begin{gather*}
 d_{r}^{n}(a\otimes c)=\sum\limits a\theta _{C,A}(c_{1,{1}})\otimes c_{2,{p-1}}.
\end{gather*}
The dif\/ferential $d^n_r$ maps $A^{n}\ot C_{p}$ to $A^{n+1}\ot C_{p-1}$. Thus $\K^\bullet_r(A,C,m)$ is a subcomplex of $\K^\bullet_r(A,C)$, for all $m\in\N$, where $\K^n_r(A,C,m)= A^{n}\ot C_{m-n}$. Note that, by convention, $C_p=0$ for $p<0$, so $\K^n_r(A,C,m)$ is trivial if either $n<-1$ or $n>m$.

Using the fact that $(A^{\rm op},C^{\rm op})$ is an almost-Koszul pair over $R^{\rm op}$, cf.\ \cite[Remark~1.4]{jps}), a~complex of left $C$-comodules is obtained by setting $\operatorname{K}_l^\bullet(A,C)=\operatorname{K}_r^\bullet(A^{\rm op},C^{\rm op})$.

By combining the complexes $\operatorname{K}_l^\bullet(A,C)$ and $\operatorname{K}_r^\bullet(A,C)$, in \cite{jps} one constructs another cochain complex $\operatorname{K}^\bullet(A,C)$ in the category of $C$-bicomodules. Since we do not use it in this paper, we omit its def\/inition.

By duality, to an almost-Koszul pair correspond also three chain complexes $\operatorname{K}^l_\bullet\!(A{,}C)$, $\operatorname{K}^r_\bullet\!(A{,}C)$ and $\operatorname{K}_\bullet(A,C)$. For instance, $\operatorname{K}^l_\bullet(A,C)$ is the complex of left $A$-modules,
\begin{gather*} \operatorname{K}_{-1}^l(A,C)=R \qquad \text{and} \qquad \operatorname{K}^l_n(A,C)=A \ot C_n,\end{gather*}
whose dif\/ferential $d_{0}^l$ is given by the left action of $A$ on $R=C_0$. For $n > 0$, the map $d_n^l$ acts as
\begin{gather*}
	d_{n}^{l}(a\otimes c)=\sum\limits a\theta _{C,A}(c_{1,{1}})\otimes c_{2,{n-1}}.
\end{gather*}
The complex $\operatorname{K}_\bullet^l(A,C)$ also decomposes as a direct sum $\oplus_{m\geq 0}\operatorname{K}^l_\bullet(A,C,m)$ of subcomplexes, where ${K}^l_n(A,C,m)=A^{m-n}\ot C_n$. Note that, ${K}^l_n(A,C,m)=0$, for $n>m$.

We conclude this subsection by recalling that the exactness of any of these six complexes implies the exactness of all the others, cf.\ \cite[Theorem~2.3]{jps}. Whenever this is the case, $(A,C)$ is called \emph{Koszul}. Thus, for a Koszul pair $(A,C)$, the complexe $\operatorname{K}_\bullet^l(A,C)$ is a resolution of $R$ by graded projective left $A$-modules. Similarly, for such a pair, $\operatorname{K}^\bullet_r(A,C)$ is a resolution of $R$ by injective graded right $C$-comodules.

The subcomplex $\K_\bullet^l(A,C,0)$ is trivial in degree $n$, for all $n\neq -1,0$ and $d_{-1}^l\colon R\to A^0\ot C_0$, coincides with the canonical isomorphism $R\cong R\ot_R R$. Thus, $\K_n^l(A,C,0)$ is always exact. We deduce that the pair $(A,C)$ is Koszul if and only if the complexes $\K_\bullet^l(A,C,m)$ are exact for all $m>0$. By duality, $(A,C)$ is Koszul if and only if $\K^\bullet_r(A,C,m)$ is Koszul for all $m>0$.

\subsection{Examples of morphisms of almost-Koszul pairs}\label{fa:exemples_morphisms}
We start by def\/ining a morphism $(\Id_A,\phi^A)$ from $(A,A^!)$ to $\big(A,T(A)\big)$. The graded coring morphism $\phi^A$ is obtained applying \cite[Proposition 1.24]{jps} as follows. We keep the notation from the above mentioned result. Deleting the component of degree $-1$ and then applying the functor $R\ot_A (-)$ to $\phi _{\bullet}\colon \K_{\bullet}^{l}(A,A^!)\to \beta_{\bullet}^{l}(A)$ we get a morphism of complexes from $R\ot_A \K_{\bullet}^{l}(A,A^!)$ to $R\ot_A \beta_{\bullet}^{l}(A)$. Note that the last two complexes are concentrated in non-negative degrees. By \cite[Proposition 1.23]{jps} the former complex is isomorphic to $(A^!,0)$, the complex with tri\-vial dif\/ferential maps and whose module of $n$-chains coincides with~$A^!_n$. On the other hand, by def\/inition, the latter complex equals $\Od{\bullet}{A}$. Since $\Id_R\ot\phi_\bullet$ is compatible with the coring structure of~$A^!$ and~$\Ou{\bullet}{A}$, by applying the homology functor we get the desired coring map $\phi^A\colon A^!\to T(A)$. Let us remark that
\begin{gather*} A^!_n:=\bigcap_{i=1}^{n-1}\big({A^1}\big)^{(i-1)}\ot\Ker\mu^{1,1}\ot \big({A^1}\big)^{(n-i-1)}\subseteq A_+^{(n)},
\end{gather*}
so any $x$ in $A^!_n$ is an $n$-cycle in $\Od{\bullet}{A}$ and $\phi^A(x)$ as the homology class of $x$. It is easy to check now that $(\Id_A,\phi^A)$ is a morphism of almost-Koszul pairs.

In a similar way one def\/ines a canonical morphism of almost Koszul-pairs $(\phi_C,\Id_C)$ from $(E(C),C)$ to $(C^!,C)$. Deleting the component of degree $-1$ of $\phi ^{\bullet}\colon \beta^\bullet_{r}(C)\to \K^\bullet_{r}(C^!,C)$ and then applying the functor $\Hom^C(R, -)$ to the resulting map of complexes we get a morphism $\ol{\phi}^{\bullet}$ from~$\Ou{\bullet}{C}$ to the cochain complex with trivial dif\/ferential maps $(C^!,0)$, cf.\ \cite[Propositions~1.23 and~1.24]{jps}. Henceforth, we can def\/ine the coring map $\phi_C\colon E(C)\to C^!$ by $\phi_C=H^\bullet(\ol{\phi}^{\bullet})$. Let us notice that, for any $x_1,\dots,x_n\in C_+$, we have
\begin{gather*}
 \ol{\phi}^{n}(x_1\ot\cdots \ot x_n)=a_1a_2\cdots a_n\in C^!_n,
\end{gather*}
where $a_i= (\theta_{C^!,C}\circ\ol{\pi}_1)(x_i)\in C^!_1$ and $\ol{\pi}_1\colon C_+\to C_1$ denotes the map induced by the projection $C\to C_1$. Thus, $\phi_C$ maps the cohomology class of an $n$-cocyle $\omega$ to $ \ol{\phi}^{n}(\omega)$.

\section[Koszul $R$-rings revisited]{Koszul $\boldsymbol{R}$-rings revisited}\label{section3}

The main goal of this section is to characterize Koszul $R$-rings in terms of properties of the $R$-coring $T(A)$. In particular, we shall recover the well known result that $A$ is Koszul if and only if $T_{n,m}(A)=0$ for $n\neq m$ (cf.~\cite[Theorem~1.9(6)]{mst}). We also include here a new proof of the fact that a Koszul $R$-ring is quadratic, cf.~\cite[Proposition~1.2.3]{bgs}. Our approach will allow us to show in the next section that similar results holds true for connected graded $R$-corings.

A comprehensive characterization of Koszul rings was given in \cite[Theorem~1.10]{mst}. For completeness, let us recall it here, in a slightly modif\/ied form, to put in evidence the interplay with Koszul pairs.

\begin{Theorem}\label{t-k7}
Let $A$ be a connected strongly graded $R$-ring. The following are equivalent:
\begin{enumerate}\itemsep=0pt
\item[$(1)$] the $R$-ring $A$ is Koszul;
\item[$(2)$] the pair $(A,T(A))$ is Koszul;
\item[$(3)$] the pair $(A,A^!)$ is Koszul;
\item[$(4)$] the canonical $R$-coring morphism $\phi^A \colon A^! \to T(A)$ is an isomorphism;
\item[$(5)$] the $R$-coring $T(A)$ is strongly graded;
\item[$(6)$] any primitive element of $T(A)$ is homogeneous of degree~$1$;
\item[$(7)$] if $n \neq m$, then $T_{n,m}(A) = 0$.	
\end{enumerate}
\end{Theorem}

Note that compared to the cited result, we have added statement (2), which is equivalent with (1), taking into account \cite[Theorem 2.13]{jps}.

Next we give a new proof of the fact that a Koszul $R$-ring is quadratic, which is based on Lemma \ref{lemator} and Proposition \ref{proptor}. These results may be of interest in their own right.

\begin{Lemma} \label{lemator}
For any connected strongly graded $R$-ring $A=\oplus_{n \geq 0} A^n$, the following sequence is exact
\begin{gather*}0 \longrightarrow\operatorname{Tor}^A_{2,m}(R,R) \rightarrow \operatorname{Tor}_{2,m}^{A/{A^{\geq m}}}(R,R)\rightarrow A^m \rightarrow 0,
\end{gather*}
where $A^{\geq m}$ denotes the ideal $\oplus_{p \geq m} A^p$ and $\operatorname{Tor}_{2,m}^{A/A^{\geq m}}(R,R)$ is the component of homological degree $m$ of $\operatorname{Tor}_2^{A/A^{\geq m}}(R,R)$.
\end{Lemma}

\begin{proof}
 Since $A$ is strongly graded, for any positive $2$-partition $\boldsymbol{m}$ of $m$, we have that $A^{\boldsymbol{m}} \xrightarrow{d_2^m} A^m$ is surjective, so the following sequence is exact:
\begin{gather} \label{surj}
0 \rightarrow \Ker d^m_2 \rightarrow \bigoplus\limits_{\bs{m}\in\mathscr{P}_2(m)}A^{\bs{m}}\rightarrow A^m \rightarrow 0.
\end{gather}
Using the complex $\Omega_\bullet(A,m)$ we get
\begin{gather} \label{tor2}
\operatorname{Tor}_{2,m}^A(R,R)=\operatorname{Ker}d_2^m/\operatorname{Im}d_3^m.
\end{gather}
On the other hand, to work with the algebra $A/A^{\geq m}$, we use the complex $\Omega_\bullet(A/A^{\geq m},m)$. Since in degree $1$ this complex is trivial, for any $m$, we have
\begin{gather} \label{tor1}
\Tor_{2,m}^{A/A^{\geq m}}(R,R)= \bigoplus\limits_{\bs{m}\in\mathscr{P}_2(m)} A^{\bs{m}}/\im d_3^m.
\end{gather}
Putting all the information together, we get the sequence
\begin{gather*}
0 \longrightarrow\Ker d^m_2/\im d_3^m \longrightarrow \bigoplus\limits_{\bs{m}\in\mathscr{P}_2(m)} A^{\bs{m}}/\im d_3^m \longrightarrow A^m \longrightarrow 0,
\end{gather*}
which is also exact due to (\ref{surj}) and gives the desired sequence, because of the relations (\ref{tor2}) and (\ref{tor1}).
\end{proof}

\begin{Proposition}\label{proptor}
Let $\pi\colon B\to A$ denote a morphism of strongly graded connected $R$-rings. If $m$ is a positive integer such that the components $\pi^i$ are bijective, for all $0\leq i <m$, and $\Tor_{2,m}^A(R,R)=0$, then $\pi^m$ is bijective as well.
\end{Proposition}

\begin{proof}
 By the relations \eqref{tor2} and \eqref{tor1}, it follows that $\pi$ induces the maps:
\begin{gather*}
 \widetilde{\pi}^m\colon \ \operatorname{Tor}_{2,m}^{B/B^{\geq m}}(R,R) \rightarrow \operatorname{Tor}_{2,m}^{A/A^{\geq m}}(R,R)_m, \qquad\overline{\pi}^m\colon \ \operatorname{Tor}^B_{2,m}(R.R) \rightarrow \operatorname{Tor}^A_{2,m}(R,R).
\end{gather*}
By construction, the squares of the following diagram are commutative,
\begin{gather*}
\xymatrixcolsep{2pc}\xymatrix{
0\ar[r]&\operatorname{Tor}^B_{2,m}(R,R) \ar[r]^-{i_B^m} \ar[d]^{\overline{\pi}^m} & \operatorname{Tor}_{2,m}^{B/B^{\geq m}}(R,R) \ar[r]^-{p_B^m} \ar[d]^{\widetilde{\pi}^m} & B^m \ar[r] \ar[d]^{\pi^m} & 0 \\
0\ar[r]&\operatorname{Tor}_{2,m}^A(R,R) \ar[r]_-{i_A^m} & \operatorname{Tor}_{2,m}^{A/A^{\geq m}}(R,R) \ar[r]_-{p_A^m} & A^m \ar[r] & 0.}
\end{gather*}
Since the homogeneous components of the graded $R$-rings $A/A^{\geq m}$ and $B/B^{\geq m}$ are zero in degree $i\geq m$, and $\pi^i$ is bijective for $i<m$, we deduce that $\pi$ induces an isomorphism $\Omega_\bullet(B/B^{\geq m})\cong\Omega_\bullet(A/A^{\geq m})$. In particular, we get that $\widetilde{\pi}^m$ is an isomorphism of $R$-bimodules. By the Snake lemma, the kernel of $\pi^m$ is isomorphic to $\Coker \ol{\pi}^m$. As $\Tor_{2,m}^A(R,R)=0$ it follows that $\pi^m$ is injective. Clearly, $\coker\pi^m=0$.
\end{proof}

To a connected graded $R$-ring $A$ we associate the $R$-ring $\langle A^1, K_A \rangle $, where $K_A$ denotes the kernel of $\m^{1,1} \colon A^1 \otimes A^1 \rightarrow A^2$. There is a unique $R$-ring morphism $\phi_A\colon \langle A^1,K_A\rangle\to A$ which lifts $\m^{1,1}$. This morphism is surjective, provided that $A$ is strongly graded. Following \cite{bgs}, we say that $A$ is \emph{quadratic} if and only if $A$ is a connected strongly graded $R$-ring such that $\phi_A$ is an isomorphism.

We now f\/ix a connected strongly graded $R$-ring $A$. Let $V=A^1$. The kernel of the canonical $R$-ring map $T_R^a(V)\to A$ will be denoted by $\widetilde{K}_A=\sum\limits_{m\in\mathbb{N}}\widetilde{K}_A^m$. Obviously, the $m$-degree component~$\widetilde{K}^m_A$ of~$\widetilde{K}_A$ coincides with the kernel of the iterated multiplication $\mu_m\colon V^{(m)}\to A^m$, and $\widetilde{K}_A^m\subseteq\gen{K_A}^m$.

Let $Z_{2,m}=\Ker d_2^m$ and $B_{2,m}=\im d_3^m$. Recall that $\{a_{\boldsymbol{m}}\}_{\boldsymbol{m}\in\mathscr{P}_2(m)}$ is a $2$-cycle if every $ a_{\boldsymbol{m}}$ is an element in $A^{\boldsymbol{m}}$ and $\sum\limits_{\boldsymbol{m}\in\mathscr{P}_2 (m)} \mu^{\boldsymbol{m}}(a_{\boldsymbol{m}})=0$. For any positive integer $m$ we set $\alpha_m(x):=\big\{ \mu_{\bs{m}}(x)\big\}_{\bs{m}\in \Pa{m}{2}}$. It is easy to see that $\alpha_m(x)$ belongs to $Z_{2,m}$, as $x$ is an element in the kernel of $\mu_m=\mu^{\bs{m}}\circ\mu_{\bs{m}}$, for any positive $2$-partition $\bs{m}$ of $m$. Thus $x\mapsto \alpha_m(x)$ def\/ines an $R$-bimodule map $\alpha_m\colon \gen{K_A}^m \to Z_{2,m}$.

Furthermore, for any element $x\in\widetilde{K}_A^m$, we have $\alpha_m(x)\in B_{2,m}$. Indeed, by def\/inition of~$\widetilde{K}_A^m$, it is enough to show that $\alpha_m(x)$ is a boundary for all $x\in V^{(i-1)}\ot K_A\ot V^{(m-i-1)}$ and $1\leq i\leq m-1$. Let $i=1$. Since $x\in K_A\ot V^{(m-2)}$, we have $\alpha_m(x)=\{x_{\bs{m}}\}_{\bs{m}\in\Pa{m}{2}}$, where
\begin{gather*}
x_{\bs{m}}=\begin{cases}
 \big(V\otimes \m(m-1)\big)(x), & \text{if} \ \bs{m}=(1,m-1), \\
 0, & \text{otherwise}.
 \end{cases}
\end{gather*}
Thus $\alpha_m(x)=-d_3^m(y)$, where $y=\big(V^{(2)}\ot\m(m-2)\big)(x)$. For all other values of $i$, the proof of the fact that $\alpha_m(x)$ is a boundary is similar, so it will be omitted.

Summarizing, $\alpha_m$ def\/ines a map from $\widetilde{K}_A^m$ to $B_{2,m}$, still denoted by $\alpha_m$. In other words, the diagram
\begin{gather*}
 \xymatrix{\widetilde{K}_A^m \ar[r]^-{\alpha_m} \ar[d] & B_{2,m} \ar[d] \\
 \ar[r]^-{\alpha_m} \gen{K_A}^m & Z_{2,m}
 }
\end{gather*}
is commutative, where the vertical arrows denote the canonical inclusions.
\begin{Lemma}
 Let $A$ be a connected strongly graded $R$-ring. If $\widetilde{K}_A^m {=} \gen{K_A}^m\!$, then $\Tor_{2,m}^A(R,R){=}0$.
\end{Lemma}
\begin{proof}
 Let $\xi= a_1\cdots a_{i}\ot a_{i+1}\cdots a_{m}$, where $a_1,\dots,a_m$ are arbitrary elements in $ V=A^1$. Since
\begin{gather*}
\xi=a_1\cdots a_{i-1}\ot a_{i}\cdots a_{m}+d_3^m( a_1\cdots a_{i-1}\ot a_{i}\ot a_{i+1}\cdots a_{m}),
\end{gather*}
we deduce by induction on $i$ that $\xi=a_1\ot a_2\cdots a_m+d_3^m(\zeta)$, for some $\zeta\in \Od{3}{A,m}$.

We now pick a $2$-cycle $x\in Z_{2,m}$. Since $A$ is strongly graded, in view of the above remarks, there exist $y\in A^1\ot A^{m-1}$ and $z\in\Od{3}{A,m}$ such that $x=y+d_3^m(z)$. Note that $y$ is a $2$-cycle as well.

Our goal is to show that $x$ is a boundary. It is enough to prove that $y\in B_{2,m}$. As $A$ is strongly graded, the map $f= V\ot \mu_{m-1}$ is surjective. Hence, there exists $y'\in V^{(m)}$ such that $f(y')=y$. Since $y$ is a cycle, we also get that $y'$ belongs to $\gen{K_A}^m$, the kernel of $\mu_m$.

By assumption, $\gen{K_A}^m=\widetilde{K}_A^m$. Thus $y'=\sum\limits_{i=1}^{m-1}y'_i$, for some $y'_i\in V^{(i-1)}\ot K_A\ot V^{(m-i-1)}$. Note that $f(y'_i)=0$, for any $i>1$. Thus $y=f(y'_1)$. As $\mu_{\bs{m}}(y'_1)=0$, for any positive $2$-partition $\bs{m}\neq(1,m-1)$, it follows that $y=\alpha_m(y'_1)$. We conclude the proof that $y$ is a boundary by remarking that $\alpha_m$ maps $\widetilde{K}_A^m$ to $B_{2,m}$.
\end{proof}

\begin{Proposition} \label{prop:quadcoring} Let $A$ be a strongly graded $R$-ring. Then $A$ is quadratic if and only if $\operatorname{Tor}_{2,m}^A(R,R)$ vanishes for all $m \geq 3$.
\end{Proposition}

\begin{proof}
Let $B:=\langle A^1,K_A\rangle$. We f\/irst prove that $A$ is a quadratic $R$-ring, provided that $\operatorname{Tor}_{2,m}^A(R,R)=0$, for all $m \geq 3$. It is enough to show that the components $\phi_A^m$ of the canonical map $\phi_A\colon B \rightarrow A$ are all isomorphisms. We proceed by induction on $m$. The maps $\phi_A^0$ and $\phi^1_A$ are obviously injective, because they coincide with the identity maps of $R$ and $A^1$, respectively. Let us assume that $\phi^k_A$ is an isomorphism, for all $k \leq m-1$, where $m\geq 3$. By assumption, $\operatorname{Tor}_{2,m}^A(R,R)=0$ for all $m\geq 3$. Hence, using Proposition \ref{proptor}, it follows that $\phi_A^m$ is bijective.

The converse follows directly by the preceding lemma, as $\widetilde{K}_A^m = \gen{K_A}^m$, for all $m$, by the def\/inition of quadratic $R$-rings.
\end{proof}

\section[Koszul $R$-corings]{Koszul $\boldsymbol{R}$-corings}\label{section4}

In this section we introduce and study the properties of Koszul corings, the dual notion of Koszul $R$-rings. Here we shall also show that any Koszul coring is quadratic.

Strongly graded comodules, the substitute for strongly graded rings when one works with graded corings instead of graded rings, will play an important role in this part of the paper. Hence, we start by discussing their main properties.

Let $(X,\rho^X)$ be a graded right comodule over a graded $R$-coring $C$. For every $n$, the map~$\rho^X$ def\/ines a morphism of graded $C$-comodules $\rho_n^X$ from $X$ to $X_n\ot C$, whose component of degree~$k$ is~$\rho_{n,k-n}^X$. Note that the component of degree $k$ of $X_n\ot C$ is $X_n\ot C_{k-n}$, where $C_i=0$, for all $i<0$. Thus, in particular, $\rho_{n,k-n}^X=0$ for $k<n$.

\begin{Definition}\label{de:strongly_graded_n}
Let $X=\oplus_{n\geq 0} X_n$ be a graded right $C$-comodule, with structure map $\rho^X$. We say that $X$ is \emph{strongly graded in degree} $n$ if the morphism of graded $C$-comodules $\rho^X_{n} \colon X \rightarrow X_n \otimes C$ is injective.
\end{Definition}

Let us remark that, by def\/inition, a graded right $C$-comodule is strongly graded in degree $n$ if and only if $X_p=0$ for all $p<n$ and $\rho_{n,q}\colon X_{n+q}\to X_n\ot C_q$ is injective for all $q\geq 0$.

Some useful properties of graded comodules are collected in the following lemma.

\begin{Lemma}\label{le:strongly_graded_n}
Let $C$ be a strongly graded $R$-coring. Let $(X,\rho^X)$ denote a graded right $C$-comodule.
\begin{enumerate}\itemsep=0pt
 \item[$1.$] The inclusion $\Ker\rho^X_{i,m-i}\subseteq \Ker\rho^X_{j,m-j}$ holds for all $0\leq j\leq i < m$.

 \item[$2.$] There is a canonical isomorphism $\Hom^C(R,X)_m\cong\bigcap_{i=0}^{m-1}\Ker\rho^X_{i,m-i}=\Ker\rho^X_{m-1,1}$.

 \item[$3.$] $X$ is strongly graded in degree $n$ if and only if $X_i=0$, for $i<n$, and $\rho_{i,1}^X$ is injective for $i\geq n$.

 \item[$4.$] Let $f\colon X\to Y$ be a morphism of graded right $C$-comodules. If $X$ is strongly graded in degree $n$ and $f_n\colon X_n\to Y_n$ is an injective map then $f$ is injective as well.
 \end{enumerate}
\end{Lemma}

 \begin{proof}
 To prove (1), consider the following commutative diagram:
 \begin{gather*}
 \xymatrixcolsep{5pc}\xymatrix{ X_m \ar[r]^-{\rho_{i,m-i}^X} \ar[d]_-{\rho_{j,m-j}^X} & X_i \ot C_{m-i} \ar[d]^-{\rho_{j,i-j}^X} \\
X_j \ot C_{m-j} \ar[r]_-{\operatorname{I}_{X_j} \otimes \Delta^C_{i-j,m-i}} & X_j \ot C_{i-j} \ot C_{m-i}.}
 \end{gather*}
Let $x \in \operatorname{Ker}\rho_{i,m-i}^X$. Using the fact that $\operatorname{I}_{X_j} \ot \Delta^C_{i-j,m-i}$ is injective, as $C$ is strongly graded, it follows that $\rho^X_{j,m-j}(x)$ is zero, hence $x \in \operatorname{Ker}\rho_{j,m-j}^X$, as required.

(2) Let $X$ and $Y$ denote two graded $C$-comodule. By def\/inition, $f\in \Hom^C(X,Y)_m$ if and only if $f$ is a morphism of $C$-comodules and $f(X_k)\subseteq Y_{m+k}$, for all $k$. In particular, for a~$C$-colinear map $f\colon R\to X$ of degree $m$, we have $f(1)\in X_m$ and $\rho^X(f(1))=f(1)\ot 1$. The latter relation is equivalent to the equations $\rho^X_{i,m-i}(f(1))=0$, for all $0\leq i\leq m-1$.
In conclusion, the required isomorphism is given by the map $f\mapsto f(1)$. The equation from the statement follows by the f\/irst part of the lemma.

(3) In view of the remark just after the Def\/inition \ref{de:strongly_graded_n}, the condition $X_i=0$ for $i<n$ is part of the hypothesis for both implications. We must prove that $\rho^X_{i,1}$ is injective for all $i\geq n$ if and only if $\rho^X_{n,q} \colon X_{n+q} \to X_n \ot C_q$ is injective for all $q\in\N$.

To prove that $\rho_{n,q}^X$ is injective for $q\in\N$ we proceed by induction on $q$. The map $\rho_{n,0}^M$ coincides with the canonical isomorphism $M_n\cong M_n\ot R$ and $\rho_{n,1}^M$ is injective by assumption. Let us assume that $\rho^X_{n,q}$ is injective. Thus, the horizontal arrow on the bottom of the following commutative diagram is injective,
\begin{gather*}
 \xymatrixcolsep{5pc} \xymatrix{ X_{n+q+1} \ar[r]^-{\rho_{n,q+1}^X} \ar[d]_{\rho_{n+q,1}^X} & X_n \ot C_{q+1} \ar[d]^-{\operatorname{I}_{X_n} \ot \Delta^C_{q,1}} \\
X_{n+q} \ot C_1 \ar[r]_-{\rho^X_{n,q} \ot \operatorname{I}_{C_1}} & X_n \ot C_q \ot C_1.}
\end{gather*}
Since, by hypothesis, the leftmost vertical map is injective it follows that $\rho^X_{n,q+1}$ is also injective.

For the converse one uses the commutative diagram
\begin{gather*}
 \xymatrixcolsep{5pc}\xymatrix{ X_{i+1} \ar[r]^-{\rho^X_{i,1}} \ar[d]_-{\rho^X_{n,i+1-n}} & X_{i} \ot C_1 \ar[d]^-{\rho^X_{n,i-n} \ot \operatorname{I}_{C_1}} \\
X_ n\ot C_{i+1-n} \ar[r]_-{\operatorname{I}_{X_n} \ot \Delta^C_{i-n,1}} & X_n \ot C_{i-n} \ot C_1.}
\end{gather*}
The component $\rho^X_{n,i+1-n}$ is injective by assumption, while $\operatorname{I}_{X_n}\ot\Delta_{i-n,1}^C$ is so as $C$ is strongly graded. We conclude that $\rho^X_{i,1}$ is injective, as required.

(4) We have to prove that all components $f_k$ are injective. For $k<n$ this property trivially holds as, by preceding part of the lemma, we have $X_k=0$. The map $f_n$ is injective by assumption, so we can suppose that $k>n$. The following diagram is commutative, as $f$ is a morphism of graded comodules,
\begin{gather*}
 \xymatrixcolsep{5pc}\xymatrix{ X_{k} \ar[r]^-{f_{k}} \ar[d]_-{\rho^X_{n,k-n}} & Y_{k} \ar[d]^-{\rho^Y_{n,k-n}} \\
X_n \ot C_{k-n} \ar[r]_-{f_n \ot \operatorname{I}_{C_{k-n}}} & Y_n \ot C_{k-n}.}
\end{gather*}
By hypothesis, $f_n\ot\Id_{C_{k-n}}$ is injective. On the other hand, $\rho_{n,k-n}^X$ is injective, as $X$ is strongly graded in degree $n$. We conclude that $x=0$.
 \end{proof}

We can now introduce Koszul corings by dualizing \cite[Def\/inition 1.2.1]{bgs}. Several characterizations of this class of graded corings are given in Theorem \ref{thm:kcoring}.

\begin{Definition} \label{def:kcor}
A connected $R$-coring $C$ is called \emph{Koszul} if $R$ has an resolution $0 \rightarrow R \rightarrow Q^\bullet$ by injective graded right $C$-comodules such that every term $Q^n$ is strongly graded in degree~$n$.
\end{Definition}

\begin{Theorem} \label{thm:kcoring}
Let $C$ be a connected, strongly graded $R$-coring. The following are equivalent:
\begin{enumerate}\itemsep=0pt
\item[$(1)$] the coring $C$ is Koszul;
\item[$(2)$] the pair $(E(C),C)$ is Koszul;
\item[$(3)$] the pair $(C^!,C)$ is Koszul;
\item[$(4)$] the canonical morphism $E(C)\rar C^!$ is an isomorphism;
\item[$(5)$] the $R$-ring $E(C)$ is strongly graded;
\item[$(6)$] the canonical map $QE(C)\to E^1(C)$ is an isomorphism;
\item[$(7)$] the relation $E^{n,m}(C)=0$ holds for all $n \neq m$.
\end{enumerate}
\end{Theorem}

\begin{proof}
Note that, for any Koszul pair $(A,C)$, the complex $\K^\bullet_r(A,C)$ is a resolution of $R$ sa\-tisfying the conditions from the def\/inition of Koszul corings. In particular it follows that $C$ is Koszul, so the implications $(2) \Longrightarrow (1)$ and $(3) \Longrightarrow (1)$ are obvious.

To prove $(1) \Longrightarrow (7)$, assume that $C$ is Koszul. That is, following the def\/inition, $R$ has an injective resolution $0 \rightarrow R \rightarrow Q^\bullet$ such that $Q^n$ is strongly graded in degree $n$. Note that $E^{n,m}(C)=\operatorname{Ext}_{n,m}^C(R,R)$ is the cohomology in degree $n$ of the complex obtained by applying the functor $\operatorname{Hom}^C(R,-)_m$ to this resolution. Thus, to conclude the proof of this implication, it is enough to show that $\operatorname{Hom}^C(R,Q^n)_m$ is trivial for $m\neq n$. By Lemma \ref{le:strongly_graded_n}(2) we have $\operatorname{Hom}^C(R,Q^n)_m=\operatorname{Ker}\rho^{Q^n}_{m-1,1}$. Thus, the claim follows from the fact that $Q^n$ is strongly graded since, in view of Lemma \ref{le:strongly_graded_n}, we have $Q^{n,m}=0$ for $m<n$, and the component $\rho^{Q^n}_{m-1,1}$ is injective for all $m>n$.

For $(7) \Longrightarrow (3)$, let us assume that $E(C)$ is diagonal. We shall prove by induction on $n$ that $\K_r^\bullet(C^!,C) $ is exact in degree $n$. The sequence
\begin{gather*}
 0\longrightarrow R\longrightarrow C_0^!\ot C\longrightarrow
 C_1^!\ot C
\end{gather*}
is exact by the def\/inition of $\K^\bullet(C^!,C)$ and the fact that $C$ is strongly graded, cf.\ \cite[Lemma~2.1]{jps}.

We now assume that $\K^\bullet(C^!,C)$ is exact in degree $i$, for all $0\leq i\leq n-1$, and prove that it is exact in degree $n$ as well. We consider the exact sequence
\begin{gather}\label{eq:sequenceC!}
 0 \rightarrow R \rightarrow C^!_0\ot C \xrightarrow{d^0_r} C_1^! \otimes C \xrightarrow{d^1_r} \cdots \xrightarrow{d^{n-2}_r} C^!_{n-1}\ot C \xrightarrow{d^{n-1}_r} C_n^! \otimes C \xrightarrow{\al}Y \rightarrow 0,
\end{gather}
where $Y:=\Coker d^{n-1}_r$ and $\al$ denotes the canonical projection. We claim that $Y$ is strongly graded in degree $n+1$. Since $d_r^{n-1}(\widehat{x}\ot c)= \widehat{x\ot c}\ot 1$, for all $x\in C_1^{n-1}$ and $c\in C_1$, the homogeneous component $(C_n^!\ot C)_n$ is included into the image of $d_r^{n-1}$. Thus $Y$ has no non-zero elements of degree $n$. On the other hand, if $Y_k$ denotes $k$-degree component of $Y$, then $Y_k=0$ for all $k<n$, as $Y$ is the quotient of $C_n^!\ot C$ by a graded subcomodule. In conclusion, for proving our claim, it remains to show that $\rho^Y_{m-1,1}$ is injective for all $m>n+1$.

Let $X:=\Ker \al=\im d^{n-1}_r$. We complete $0 \rightarrow {X} \rightarrow C_n^! \otimes C$ to a resolution
\begin{gather*}
 0 \rightarrow X\rightarrow C_n^!\ot C \rightarrow Q^{n+1} \rightarrow Q^{n+2} \rightarrow \cdots
\end{gather*}
by injective graded right $C$-comodules, which combined with \eqref{eq:sequenceC!} yields us a new injective resolution
\begin{gather*} 
0 \rightarrow R \rightarrow C_0^!\ot C \rightarrow C_1^!\ot C\rightarrow \dots \rightarrow C_n^! \otimes C \rightarrow Q^{n+1} \rightarrow Q^{n+2} \rightarrow \cdots.
\end{gather*}
Thus we can compute $\operatorname{Ext}_C^{n+1,m}(R,R)$ as the cohomology in degree $n+1$ of the complex:
\begin{gather*}
 0 \rightarrow \operatorname{Hom}^C(R,C_0^!\ot C)_m \rightarrow \operatorname{Hom}^C(R,C_1^!\ot C)_m \rightarrow \dots \rightarrow \operatorname{Hom}^C(R, C_n^!\ot C)_m\\
 \hphantom{0}{} \rightarrow \operatorname{Hom}^C\big(R,Q^{n+1}\big)_m \rightarrow \cdots.
\end{gather*}
Since $\operatorname{Ext}^{1,m}_C(R,X)$ is the $1$-st cohomology group of the complex
\begin{gather*}
 0 \rightarrow \operatorname{Hom}^C(R,C_n^! \otimes C)_m \rightarrow \operatorname{Hom}^C(R,Q^{n+1})_m \rightarrow \operatorname{Hom}^C\big(R,Q^{n+2}\big)_m \rightarrow \cdots
\end{gather*}
we deduce that $\operatorname{Ext}^{1,m}_C(R,X)\cong\operatorname{Ext}_C^{n+1,m}(R,R)$. Recall that $X$ is the kernel of $\al$, so the long exact sequence connecting the functors of $\Ext^{\bullet,m}_C(R,-)_m$ can be written as follows
\begin{gather*}
 0 \rightarrow \operatorname{Hom}^C(R,X)_m \rightarrow \operatorname{Hom}^C(R, C_n^! \otimes C)_m \rightarrow \operatorname{Hom}^C(R,Y)_m \rightarrow \operatorname{Ext}^{1,m}_C(R,X) \rightarrow \cdots.
\end{gather*}
From the standing assumption, $E(C)$ is diagonal, so using the above isomorphism between the Ext-groups, it follows that $\operatorname{Ext}_C^{1,m}(R,X)=0$, for all $m>n+1$. Moreover, $\operatorname{Hom}^C(R,C_n^! \otimes C)$ is isomorphic to $\operatorname{Hom}_R(R,C_n^! )$ as graded $R$-modules, where the latter module is concentrated in degree~$n$. It follows that $\operatorname{Hom}^C(R,Y)_m=0$, for all $m > n+1$. By Lemma~\ref{le:strongly_graded_n}(2) it follows that $\rho^Y_{m-1,1}$ is injective for all $m>n+1$. Since we already know that $Y_k=0$ for $k<n+1$, we conclude by Lemma~\ref{le:strongly_graded_n}(3) that $Y$ is strongly graded in degree $n+1$, that is our claim has been proved.

We can now prove that $\K^\bullet(C^!,C)$ is exact in degree $n$. We must show that the kernel of the map $\pa\colon Y\to C^!_{n+1}\ot C$ induced by $d^n_r$ is trivial. In view of Lemma~\ref{le:strongly_graded_n}(4), as $Y$ is strongly graded in $n+1$, it is enough to prove that the component $\pa_{n+1}\colon Y_{n+1}\to C_{n+1}\ot C_0$ of $\pa$ is injective.

Let $y$ be an element in the kernel of $\pa_{n+1}$. Hence there are ${y}_1, \dots,{y}_p\in C^{(n)}$ and $c_1,\dots, c_p\in C_1$ such that $y$ is the equivalence class of $\sum\limits_{i=1}^p \widehat{y}_i\ot c_i$ in $Y_{n+1}=(C_n^!\ot C_1)/d_r^{n-1}(C_{n-1}^!\ot C_2)$. Here $\widehat{y}_i$ denotes the class of $y_i$ in $C^!_n=C_1^{(n)}/W_n$, where $W_n=\sum\limits_{k=1}^{n-1} C_1^{(k-1)}\ot\im\Delta_{1,1}\ot C_1^{(n-k-1)}$. Since $\pa_{n+1}(y)=0$ and using the fact that $d_r^n$ is induced by the multiplication in $C^!$, it follows that $y'=\sum\limits_{i=1}^p {y}_i\ot c_i$ belongs to the submodule $W_{n+1}\subseteq C^{(n+1)}$ from the construction of $C_{n+1}^!$. Note that $W_2=\im\Delta_{1,1}$. Thus $y'=x+\sum\limits_{j=1}^q y_j''\ot c^j_{1,1}\ot c^j_{2,1}$, for some $x\in W_n\ot C_1$, $y''_1,\dots,y''_q\in C_1^{(n-1)}$ and $c^1,\dots,c^q\in C_2$. The relation $y=0$ now follows by the computation
\begin{gather*}
 \sum_{i=1}^p \widehat{y}_i\ot c_i=\sum_{j=1}^q \widehat{y_j''\ot c^j_{1,1}}\ot c^j_{2,1}=d_r^{n}\left(\sum_{j=1}^q {\widehat{y_j''}\ot c^j}\right).
\end{gather*}

For the implication $(3) \Longrightarrow (4)$ we use the morphism of complexes $\phi^\bullet\colon \beta^\bullet_r(C) \rar \operatorname{K}_r^\bullet(C^!,C)$, which lifts the identity of $R$, see Section~\ref{fa:exemples_morphisms}. Since both $\K^{\bullet}_{r}(C^!,C)$ and $\beta{}^{\bullet}_{r}(C)$ are resolutions of $R$ by injective right $C$-comodules, by the comparison theorem, the morphism $\{\phi^n\}_{n\in\N}$ is invertible up to a homotopy in the category of complexes of right $C$-comodules. Hence $\Hom^C(R,\phi^\bullet)$ is invertible up to a homotopy in the category of complexes of right $R$-modules. It follows that $\h^n(\Hom^C(R,\phi^\bullet))$ is an isomorphism for all $n\geq 0$. Thus $\phi^n_C\colon E_n(C)\to C^!_n$ is an isomorphism, as it coincides with $\h^n(\Hom^C(R,\phi^\bullet))$, see Section~\ref{fa:exemples_morphisms}.

The implication $(4) \Longrightarrow (5)$ follows immediately since $C^!$ is always strongly graded, hence $E(C)$ is so. Furthermore, $(5) \iff (6)$ and $(5) \Longrightarrow (7)$ are direct consequences of Lemma~\ref{lema:ring}.

We conclude the proof by remarking, for the implication $(3) \Longrightarrow (2)$, that the complex $\K^\bullet_r(C^!,C)$ is isomorphic to $\K^\bullet_r(E(C),C)$, which makes the latter exact.
\end{proof}

\begin{Remark}
 By the proof of \cite[Theorem~2.14]{jps}, the complex $\K_\bullet^l(A,A^!)$ is isomorphic to the Koszul complex of $A$, which was introduced in \cite[p.~483]{bgs}. By analogy, $\operatorname{K}^\bullet(C^!,C)$ will be seen as the Koszul complex associated to a connected $R$-coring $C$.
\end{Remark}

Using the method from the previous section we are going to show that any Koszul coring is quadratic. This property of Koszul corings will follow as a direct application of a cohomological criterion for bijectivity of a morphism between two strongly graded corings.

\begin{Lemma} \label{lemmaext}
Let $C$ be a strongly graded $R$-coring. Then the following sequence is exact:
\begin{gather*}
0 \longrightarrow C_m \longrightarrow \operatorname{Ext}^{2,m}_{C_{<m}} (R,R) \longrightarrow \operatorname{Ext}^{2,m}_C(R,R) \longrightarrow 0.
\end{gather*}
\end{Lemma}

\begin{proof} The dif\/ferential $d_m^1$ of $\Omega^\bullet(C,m)$ maps $c \in C_m$ to the family $\{\Delta_{\bs{m}}(c) \}_{\bs{m}\in\mathscr{P}_2(m)}$ where, by notation, $\Delta_{\bs{m}}$ is the component $\Delta_{m_1,m_2}\colon C_{m_1+m_2}\to C_{m_1}\ot C_{m_2}$, for any positive $2$-partition $\bs{m}=(m_1,m_2)$ of $m$. By hypothesis $\Delta_{\bs{m}}$ is injective for all $\bs{m}\in\mathscr{P}_2(m)$, so $d_m^1$ is injective as well. Thus the sequence
\begin{gather*}
 0 \longrightarrow C_m \xrightarrow{d_m^1} \Ker d_m^2 \longrightarrow \Ker d_m^2/\im d_m^1 \longrightarrow 0
\end{gather*}
is exact. Note that $\operatorname{Ext}^{2,m}_C(R,R) =\Ker d_m^2/\im d_m^1$. On the other hand, the complex $\Omega^\bullet(C_{<m},m)$ coincides in small degrees with
\begin{gather*}
0 \longrightarrow 0 \xrightarrow{d_m^0} 0 \xrightarrow{d_m^1} \bigoplus \limits_{\boldsymbol{m}_2\in\mathscr{P}_2(m)} C_{\boldsymbol{m}_2} \xrightarrow{d_m^2}  \bigoplus \limits_{\boldsymbol{m}_3\in\mathscr{P}_{3}(m)} C_{\boldsymbol{m}_3} ,
\end{gather*}
so we conclude the proof by remarking that $\Ext^{2,m}_{C_{<m}}(R,R) \simeq \Ker d_m^2.$
\end{proof}

\begin{Proposition}\label{prop:injectiv}
Let $\pi\colon C\to D$ denote a morphism of strongly graded connected $R$-rings. If~$m$ is a positive integer such that $\Ext^{2,m}_C(R,R)=0$ and the components $\pi_i$ are bijective, for all $0\leq i <m$, then $\pi_m$ is bijective as well.
\end{Proposition}

\begin{proof}
As in the case of graded $R$-rings, $\pi$ induces morphisms
\begin{gather*}
\ol{\pi}^m\colon \Ext^{2,m}_C(R,R)\to \Ext^{2,m}_D(R,R) \qquad \text{and} \qquad \widetilde{\pi}^m\colon \Ext^{2,m}_{C_{<m}}(R,R)\to \Ext^{2,m}_{D_{<m}}(R,R),
\end{gather*} such that the following diagram is commutative,
 \begin{gather*}
 \xymatrixcolsep{2pc}\xymatrix{ 0 \ar[r] & C_m \ar[d]^{\pi_m} \ar[r] & \operatorname{Ext}_{C_{<m}}^{2,m}(R,R) \ar[d]^-{\widetilde{\pi}_m} \ar[r] & \operatorname{Ext}^{2,m}_C(R,R) \ar[d]^-{\overline{\pi}_m} \ar[r]& 0 \\
0 \ar[r] & D_m \ar[r] & \operatorname{Ext}_{D_{<m}}^{2,m}(R,R) \ar[r] & \operatorname{Ext}^{2,m}_D(R,R) \ar[r]& 0.
}
 \end{gather*}
The map $\widetilde{\pi}_m$ is an isomorphism, as $\pi_i$ is bijective, for all $i<m$. Therefore, by Snake lemma, we have $\Ker\pi_m=0$ and $\coker\pi_m\cong\Ker\ol{\pi}_m=0$.
\end{proof}

\subsection{Quadratic corings}\label{fa:quadratic_coring}
For a connected $R$-coring $C$, the family $\{\Delta(n)\}_{n\in\N}$ of iterated comultiplications $\Delta(n)\colon C_n\to C^{(n)}_1$ def\/ines a morphism $\phi_C\colon C\to T_R^c(C_1)$. The image of $\phi_C$ is a subcoring $\ol{C}$ of $\widetilde{C}:=\br{C_1,\im \Delta_{1,1} }$. Hence we may regard $\phi_C$ as a coring morphism from~$C$ to~$\widetilde{C}$. We shall say that $C$ is \emph{quadratic} if and only if this morphism is a bijection. Therefore, $C$~is quadratic if and only if it is strongly graded and $\ol{C}=\widetilde{C}$.

Let $C$ be a strongly graded $R$-coring. As in the case of $R$-rings, we can relate $Z^{2,m}=\Ker d^2_m$, the set of all $2$-cocycles in $\Ou{\bullet}{C,m}$, and $\widetilde{C}_m$. To do this we f\/irst notice that, for any $2$-cocycle $x=\br{x_{\bs{m}}}_{\bs{m}\in\Pa{m}{2}}$, the element $\Delta(\bs{m})(x_{\bs{m}})$ does not depend on $\bs{m}\in\Pa{m}{2}$, as the components of the (iterated) comultiplication are all injective. Recall that $\Delta(\bs{m})=\Delta({m_1})\ot\Delta({m_2})$, for any positive $2$-partition $\bs{m}=(m_1,m_2)$. Since $\Delta(m)=\Delta(\bs{m})\circ \Delta_{\bs{m}}$, for any positive $2$-partition of $m$, it is not dif\/f\/icult to see that $\Delta(\bs{m})(x_{\bs{m}})\in \widetilde{C}_m$. Hence, we can def\/ine the function $\alpha^m\colon Z^{2,m}\to \widetilde{C}_m$ by $\alpha^m(x)=\Delta(\bs{m})(x_{\bs{m}})$.

Let $B^{2,m}:=\im d^1_m$ denote the set of $2$-coboundaries. If $x=\br{x_{\bs{m}}}_{\bs{m}\in\Pa{m}{2}}$ is a $2$-coboundary, then there is $c\in C_m$ such that $x_{\bs{m}}=\Delta^{\bs{m}}(c)$, for all positive $2$-partitions $\bs{m}$. Thus $\alpha^m(x)\in\ol{C}_m$. Henceforth, $\alpha^m$ induces a map, still denoted by $\alpha^m$, from $B^{2,m}$ to $\ol{C}_m$. We get the following commutative diagram{\samepage
\begin{gather*}
 \xymatrix{B^{2,m} \ar[r]^-{\alpha^m} \ar[d] & \ol{C}_m \ar[d] \\
 \ar[r]^-{\alpha^m} Z^{2,m} & \widetilde{C}_m,
 }
\end{gather*}
where the vertical arrows represent the canonical inclusions.}

\begin{Lemma}
 Let $C$ be a strongly graded $R$-coring. If $\ol{C}_m = \widetilde{C}_m$, then $\Ext^{2,m}_C(R,R)=0$.
\end{Lemma}

\begin{proof}
 If $x=\br{x_{\bs{m}}}_{\bs{m}\in\Pa{m}{2}}$ is a $2$-cocycle, then $\alpha^m(x)\in \widetilde{C}_m=\ol{C}_m$. Since $\ol{C}_m$ is the image of~$\Delta(m)$, there is $c\in C_m$ such that $\Delta(\bs{m})(x_{\bs{m}})=\Delta(m)(c)$. As $\Delta(m)=\Delta(\bs{m})\circ \Delta_{\bs{m}}$ and~$\Delta(\bs{m})$ is injective, we deduce that $x_{\bs{m}}= \Delta_{\bs{m}}(c)$, for any positive $2$-partition~$\bs{m}$. Thus $x$ is a~$2$-coboundary.
\end{proof}

\begin{Proposition} \label{prop:quadring} A strongly graded $R$-coring $C$ is quadratic if and only if $\operatorname{Ext}^{2,m}_C(R,R) =0$, for all $m \geq 3$.
\end{Proposition}

\begin{proof} In view of the preceding Lemma, the vanishing of $\operatorname{Ext}^{2,m}_C(R,R)$ is a necessary condition. To prove that it is a suf\/f\/icient condition as well, using Proposition \ref{prop:injectiv} and proceeding as in the proof of Proposition~\ref{prop:quadcoring}, one shows by induction that the components of $\phi_C\colon C\to\widetilde{C}$ are all isomorphisms.
\end{proof}

\begin{Corollary} \label{co:cquad}
Any Koszul $R$-coring $C$ is quadratic.
\end{Corollary}
\begin{proof}
The conclusion follows by combining the results in Proposition \ref{prop:quadring} and the fact the statements $(1)$ and $(7)$ of Theorem~\ref{thm:kcoring} are equivalent.
\end{proof}

Finally, let us add that a direct (co)product of two Koszul (co)rings is still Koszul. For $R$-rings the result was used in \cite[Section~2.10]{mst} to ``patch'' tiles forming a Koszul poset. We restate and prove the proposition using Koszul pairs, obtaining at the same time a similar result for Koszul corings.

Let $R$ and $S$ be semisimple rings. We assume that $A$ is a connected $R$-ring, $B$ is a connected $S$-ring, $C$ is a connected $R$-coring and $D$ is a connected $S$-coring. We also suppose that $(A,C)$ and $(B,D)$ are almost-Koszul pairs with respect to the $R$-bilinear map $\theta_{C,A}\colon C_1\to A^1$ and the $S$-bilinear map $\theta_{D,B}\colon D_1\to B^1$, respectively.

For an $R$-bimodule $V$ and an $S$-bimodule $W$ we shall regard the abelian group $V\oplus W$ as an $R\times S$-bimodule with respect to the actions: $(r,s)\cdot(v,w)=(rv,sw)$ and $(v,w)\cdot(r',s')=(vr',ws')$.
Note that we have an isomorphism of left $A\times B$-modules
\begin{gather}\label{izo:RxS}
 (A\times B)\ot_{R\times S}(V\oplus W)\cong (A\otimes_R V)\oplus(B\ot_S W),
\end{gather}
where the $A\times B$-action on the direct sum of $A\otimes_R V$ and $B\ot_S W$ is also def\/ined component-wise. This isomorphism is def\/ined by the map $(a,b)\ot_{R\times S}(v,w) \mapsto(a\ot_R v, b\ot_S w)$.

Recall that the comultiplication of the coproduct $C\oplus D$ is given by
\begin{gather*}
 \Delta_{p,q}(c,d)=\sum(c_{1,p},0)\ot_{R\times S}(c_{2,q},0)+\sum (0,d_{1,p})\ot_{R\times S}(0,d_{2,q}).
\end{gather*}
Let $\theta':=\theta_{C,A}$ and $\theta'':=\theta_{D,B}$. We def\/ine $\theta\colon C_1\oplus D_1\to A^1\times B^1$ given by $\theta(c,d)=\big(\theta'(c),\theta''(d)\big)$.
\begin{Proposition} \label{suma_koszul}
The pair $(A\times B,C\oplus D)$ is almost-Koszul with respect to the isomor\-phism~$\theta$. Moreover, this pair is Koszul if and only if the pairs $(A,C)$ and $(B,D)$ are so.
\end{Proposition}

\begin{proof}
To show that $(A\times B,C\oplus D)$ is almost-Koszul we take $(c,d)\in C_2\oplus D_2 $. We have
\begin{gather*}
 \big(\mu^{1,1}\circ (\theta\ot_{R\times S}\theta)\circ\Delta_{1,1}\big)(c,d)=\sum\big(\theta'(c_{1,1})\theta'( c_{2,1}),0\big)+\sum \big(0,\theta''( d_{1,1})\theta''(d_{2,1})\big)=0.
\end{gather*}
Using the identif\/ication \eqref{izo:RxS} and performing a similar computation as above one can show easily that
\begin{gather*}
 \K_\bullet^l(A\times B, C\oplus D)\cong \K_\bullet^l(A,C)\oplus \K_\bullet^l(B,D).
\end{gather*}
Thus, obviously, $(A \times B, C \oplus D)$ is Koszul if and only if both $(A,C)$ and $(B,D)$ are so.
\end{proof}

\begin{Remark}
We have seen that a connected $R$-ring $A$ is Koszul if and only if there is a~connected graded $R$-coring $C$ such that $(A,C)$ is Koszul. Similarly, a given connected $R$-co\-ring~$C$ is Koszul if and only if there is a connected graded $R$-ring $A$ such that $(A,C)$ is Koszul. In view of the preceding proposition, we deduce that the connected graded $R\times S$-ring $A\times B$ is Koszul if and only if $A$ and $B$ are Koszul too. Similarly, the connected graded $R\times S$-co\-ring~$C\times D$ is Koszul if and only if~$C$ and~$D$ are so.
\end{Remark}

\section{Locally f\/inite Koszul (co)rings}\label{section5}

In this section, we provide a f\/irst application of the results developed so far. More precisely, we shall prove that a left (right) locally f\/inite $R$-ring is Koszul if and only if its left (right) graded dual $R$-coring is Koszul as well. This result will allow us to show in the subsequent section that in a more particular case, the incidence $R$-ring of a graded f\/inite poset is Koszul if and only if its incidence $R$-coring is also Koszul.

Recall that a graded left $R$-bimodule $V=\oplus_{n\in\N} V_n$ is called \emph{left $($right$)$ locally finite} if and only if its components $V_n$ are f\/initely generated as left (right) modules. In the case when $V$ is both left and right locally f\/inite, we will simply say that $V$ is \emph{locally finite}. Throughout this section we keep the assumption that $R$ is a semisimple ring, so all $R$-bimodules are projective as left (and right) $R$-modules.

\subsection[The left dual of an $R$-bimodule]{The left dual of an $\boldsymbol{R}$-bimodule}
\looseness=-1 Let $V$ be an $R$-bimodule. Def\/ine the left dual of $V$ by $^*V=\operatorname{Hom}_R(_RV,{}_RR)$. This becomes a~bi\-module over the opposed ring $R^{\rm op}$ with respect to the actions def\/ined, for $ r\in R$ and $\alpha \in {^*V}$, by
\begin{gather*} (r\cdot \alpha)(v)=\alpha(v) r \qquad \text{and} \qquad (\alpha \cdot r)(v)=\alpha(vr).\end{gather*}

In the case that $V=\oplus_{n\in N}V_n $ is a graded $R$-bimodule we def\/ine the \textit{left graded dual} of $V$ to be the $R^{\rm op}$-bimodule $^{*\text{-gr}}V:=\oplus_{n\in N}{^*V_n}$.

It is known that the dual of the tensor product of two f\/inite-dimensional vector spaces is the tensor product of their duals (see, e.g., \cite[Section~2.7]{bgs}). A similar property holds for $R$-bimodules. More precisely, if $V$ and $W$ are $R$-bimodules, there exists a bi-additive map
\begin{gather*}
\phi'\colon \ {^*V} \times {^*W} \rar {^*(V \ot_R W)}, \qquad \phi'(\alpha , \beta)(v\ot_R w)=\alpha(v\beta(w)).
\end{gather*}
Furthermore, we have $\phi'((\alpha \cdot r) \ot_{R^{\rm op}} \beta)(v \ot_R w)=(\alpha \cdot r)(v\beta(w))=\alpha(rv\beta(w))=\phi(\alpha \ot_{R^{\rm op}} (r \cdot \beta))(v \ot_R w)$. Thus $\phi'$ is $R^{\rm op}$-balanced, so it induces a morphism of abelian groups
\begin{gather*}
\phi\colon \ {^*V} \ot_{R^{\rm op}} {^*W} \rar {^*(V \ot_R W)}, \qquad \phi(\alpha \ot_{R^{\rm op}} \beta)(v\ot_R w)=\alpha(v\beta(w)).
\end{gather*}
As a matter of fact, $\phi$ is an $R^{\rm op}$-bimodule map, as $\phi((r \cdot \alpha) \ot_{R^{\rm op}} \beta)(v \ot_R w)=(r\cdot \alpha)(v\beta(w))=\alpha(v\beta(w)) r.$ On the other hand, $[r \cdot \phi(\alpha \ot_{R^{\rm op}} \beta)](v \ot_R w)=\phi(\alpha \ot_{R^{\rm op}} \beta)(v \ot_R w)r=\alpha(v\beta(w))r$. Right-linearity is proved analogously.

We claim that, under the additional assumption that $W$ is f\/initely generated as a left $R$-module, the map $\phi$ is a bijection. Indeed, since $R$ is semisimple, $W$ is projective as a left $R$-module. So, there are f\/inite dual bases on $W$, that is two sets $\{w_1,\dots,w_n\}\subseteq W$ and $\{{^*}w_1,\dots,{^*}w_n\}\subseteq {^*}W$ such that $w=\sum\limits_{i=1}^n{^*}w_i(w)w_i$, for all $w\in W$.
 Let $\psi\colon {}^*(V\ot_R W) \to {}^*V\ot_{R^{\rm op}}{^*W} $ be the map given by
\begin{gather*}
\psi(\gamma) = \sum_{i=1}^n \gamma(- \ot_R w_i) \ot_{R^{\rm op}}{}^*w_i.
\end{gather*}
In the above formula $\gamma$ denotes an element in ${^*}(V \ot_R W)$ and the application $\gamma(- \ot_R w_i) \colon V \rar R$ acts as $v \mapsto \gamma(v \ot_R w_i)$. The only thing left to show is that $\psi$ and $\phi$ are mutual inverses.

Let $ \gamma =:\phi(\alpha \ot_{R^{\rm op}} \beta)$, for some $\alpha\in{}^*V$ and $\beta\in{}^*W$. Thus $\gamma(v \ot_R w)=\alpha(v\beta(w))$, for all $v\in V$ and $w\in W$. Hence, by the def\/inition of $\psi$, we get
\begin{gather*}
 \psi(\gamma) = \sum_{i=1}^n \gamma(-\ot w_i)\ot_{R^{\rm op}}{^*}w_i = \sum_{i=1}^n \alpha(-\beta(w_i))\ot_{R^{\rm op}}{^*}w_i = \sum_{i=1}^n\alpha\cdot \beta(w_i)\ot_{R^{\rm op}}{^*}w_i.
\end{gather*}
By the def\/inition of the left $R$-action on $^*W$ and the def\/inition of dual bases, together with the fact that $\beta$ is a morphism of left $R$-modules, we have $\beta=\sum\limits_{i=1}^n\beta(w_i)\cdot {^*}w_i$. Hence
\begin{gather*}
 \psi(\gamma)=\alpha \ot_{R^{\rm op}}\left(\sum_{i=1}^n\beta(w_i)\cdot {^*}w_i\right) =\alpha\ot \beta,
\end{gather*}
so $\psi$ is a left inverse of $\phi$.
On the other hand, if $\gamma$ belongs to the left dual of $V\ot_RW$ then
\begin{gather*}
[\phi(\psi(\gamma))](v \ot_R w) = \sum_{i=1}^n\phi\big(\gamma(-\ot_R w_i)\ot_{R^{\rm op}} {^*}w_i)\big)(v \ot w) = \sum_{i=1}^n\gamma(-\ot_R w_i)\big(v{}^*w_i(w)\big).
\end{gather*}
Thus the computation below implies that $\psi$ is a right inverse of $\phi$ as well
\begin{gather*}
\sum_{i=1}^n\gamma(-\ot_R w_i)\big(v{}^*w_i(w)\big)=\sum_{i=1}^n\gamma(v{}^*w_i(w)\ot_Rw_i)=\gamma(v\ot_R \sum_{i=1}^n{}^*w_i(w)w_i)=\gamma(v\ot_R w) .
\end{gather*}

\subsection[The graded dual of an $R$-(co)ring]{The graded dual of an $\boldsymbol{R}$-(co)ring} \label{sec:gr-dual} The left (right) dual of a left (right) f\/initely generated $R$-ring was f\/irst introduced in \cite[Section~17.9]{brw}. This construction can be easily adapted for a left locally f\/inite connected graded $R$-ring $A=\op_n A^n$. First, we def\/ine ${^{*\text{-gr}}}\! A = \op_n {^*}(A^n)$, which is an $R^{\rm op}$-bimodule with respect to the actions def\/ined in the previous subsection. When there is no risk of confusion, we shall drop the parentheses to avoid unnecessary clutter. As such, we will write ${^*}\!A^n$ instead of ${^*}(A^n)$.

Furthermore, for making ${^{*\text{-gr}}}\!A$ a graded connected $R^{\rm op}$-coring, we consider the following diagram:
\begin{gather*}
\xymatrixcolsep{3pc}\xymatrix{ {^*}\!A^{n+m} \ar@{=}[r] \ar[d]_-{{^*}\mu_{n,m}} & {^*}\!A^{n+m} \ar@{.>}[d]^-{\Delta_{n,m}} \notag \\ {^*}(A^n \ot_R A^m) \ar[r]_-\psi & {^*}\!A^n \ot_{R^{\rm op}} {^*}\!A^m.}
\end{gather*}
The leftmost vertical arrow denotes the transposed map of $\mu^{n,m} \colon A^n \ot A^m \rar A^{n+m}$, the component of the multiplication of $A$. The lower morphism $\psi$ is the isomorphism described in the previous subsection. Then $\Delta_{n,m}:=\psi \circ {^*}\mu_{n,m}$ is a morphism of $R^{\rm op}$-bimodules and the family $\{\Delta_{n,m}\}_{n,m\in\N}$ induces a map $\Delta \colon {^{*\text{-gr}}}\!A \rar {^{*\text{-gr}}}\!A \ot {^{*\text{-gr}}}\!A$ that respects the gradings on $ {^{*\text{-gr}}}\!A$ and ${^{*\text{-gr}}}\!A \ot {^{*\text{-gr}}}\!A$.

Let $ \alpha \in {^*}\!A^{n+m}$. One can show that the relation
\begin{gather*}
\Delta_{n,m}(\alpha) = \sum_{i=1}^p \alpha'_{i} \ot_{R^{\rm op}} \alpha_{i}''
\end{gather*}
holds true for some $\al'_1,\dots,\al'_p$ and $\al''_1,\dots,\al''_p$ if and only if we have
\begin{gather}
 \alpha(a'a'') =\sum_{i=1}^p \alpha_{i}'(a'\alpha_{i}''(a'')),\label{eq:Delta*-2}
\end{gather}
for all $a' \in A^n$ and $a'' \in A^m$. Using this equivalence it is easy to see now that $\Delta$ def\/ines a~coassociative comultiplication on the left graded dual of $A$, which respects the grading on~$^{*\text{-gr}}\!A$. Let us note that we have an isomorphism of rings $\operatorname{Hom}_R({_R}R,{_R}R)\cong R^{\rm op}$, so we can identify~${^*}\!A^0$ and~$R^{\rm op}$ as $R^{\rm op}$-bimodules. This isomorphism can be extended in a unique way to an $R^{\rm op}$-bimodule morphism $\varepsilon\colon {^{*\text{-gr}}\!A}\to R^{\rm op}$ so that it vanishes on all other homogeneous components of~$^{*\text{-gr}}\! A$. Obviously, $({}^{*\text{-gr}}\! A,\Delta, \varepsilon)$ is a connected graded $R^{\rm op}$-coring, which will be called the \textit{graded left dual $R^{\rm op}$-coring} of $A$. For the comultiplication of $^{*\text{-gr}}\! A$ we will use the Sweedler type notation
\begin{gather}\label{eq:sweedler}
 \Delta_{n,m}(\alpha)=\sum \alpha_{1,n}\ot\alpha_{2,m}.
\end{gather}
Thus, for $\al\in {}^\ast\!A^{n+m}$, $a' \in A^n$ and $a'' \in A^m$, the relation \eqref{eq:Delta*-2} can be rewritten as
\begin{gather}\label{eq:Delta*-3}
 \alpha(a'a'') = \sum \alpha_{1,n}(a'\alpha_{2,m}(a'')).
\end{gather}
Dually, to any graded connected $R$-coring $C$ corresponds a graded connected $R^{\rm op}$-ring $^{*\text{-gr}}C$ that we will call the \emph{graded left dual of} $C$. By def\/inition we have $(^{*\text{-gr}}C)^n={}^*(C_n)$. To simplify the notation we shall write $^*C_n$ instead of $^*(C_n)$. The \emph{graded convolution product} of $\alpha\in {^*C}_n$ and $\beta\in {^*C}_m$ is given by the relation $\alpha*\beta=\mu^{n,m}(\alpha \ot \beta)$. Hence, for $c\in C_{n+m}$, we have
\begin{gather*}
(\alpha*\beta)(c)=\sum\alpha\big(c_{1,n}\beta(c_{2,m})\big).
\end{gather*}
The unit of the graded left dual of $C$ coincides with the counit of $C$. Note that the graded left dual makes sense for a graded $R$-coring $C$ which is not necessarily left locally f\/inite.

In a similar way we can def\/ine the graded right dual $A^{{*\text{-gr}}}$ of a right locally f\/inite connected graded $R$-ring $A$ and the graded right dual $C^{*\text{-gr}}$ of a right locally f\/inite connected graded $R$-coring $C$. One can prove that, for a locally f\/inite $R$-ring $A$, there are canonical isomorphisms
$(^{*\text{-gr}}\!A )^{*\text{-gr}}\cong A\cong {}^{*\text{-gr}}(A^{*\text{-gr}}) $ of graded $R$-rings. Similar isomorphisms can be proved for a locally f\/inite $R$-coring $C$.

In order to investigate the Koszulity of the dual of a locally f\/inite $R$-ring we shall use once again almost-Koszul pairs. More precisely, we have the following result.

\begin{Proposition} \label{prop:a-ksz-dual}
Let $(A,C)$ be an almost-Koszul pair. If $A$ and $C$ are left locally finite, then the pair $({^{*\text{\rm -gr}}}C, {^{*\text{\rm -gr}}}\!A)$ is almost-Koszul. Similarly, if $A$ and $C$ are right locally finite then $(C^{*\text{\rm -gr}},A^{*\text{\rm -gr}})$ is almost-Koszul.
\end{Proposition}

\begin{proof}
Let $\theta:=\theta_{C,A}$ denote the isomorphism from the def\/inition of almost-Koszul pairs. We claim that $(^{*\text{-gr}}\! A,{}^{*\text{-gr}}C)$ is almost-Koszul with respect to the transposed map $^*\theta\colon {}^\ast\! A^2\to{}^\ast C_2$. Clearly, $^*\theta$ is an isomorphism of $R^{\rm op}$-bimodules, so we have to prove equation \eqref{ec:qksz}. For $\al\in{}^*\!A^2$ and $c\in C_2$ we have
\begin{gather*}
 \sum {}\big(^*\theta(\al_{1,1})*{}^*\theta(\al_{2,1})\big)(c) = \sum {}^*\theta(\al_{1,1})\big(c_{1,1}{}^*\theta(\al_{2,1})(c_{2,1})\big)\\
 \hphantom{\sum {}\big(^*\theta(\al_{1,1})*{}^*\theta(\al_{2,1})\big)(c)}{} =\sum {}\al_{1,1}\big(\theta(c_{1,1})\al_{2,1}(\theta(c_{2,1}))\big)
 =\sum {}\al(\theta(c_{1,1})\theta(c_{2,1}))=0.
\end{gather*}
Note that the f\/irst and the second equalities follow by the def\/initions of the product of $^{*\text{-gr}}C$ and of the transposed map, respectively. Taking into account the equivalence between the rela\-tions~\eqref{eq:sweedler} and~\eqref{eq:Delta*-3} we get the third equality. The ultimate equation holds as $\theta$ satisf\/ies the rela\-tion~\eqref{ec:qksz}.

The fact that $(A^{*\text{-gr}},C^{*\text{-gr}})$ is almost-Koszul can be proved in a similar way.
\end{proof}

We can take this result a step forward and prove that a Koszul pair corresponds by left duality to a Koszul pair.

\begin{Theorem} \label{te:kdual}
Let $(A,C)$ be a Koszul pair. If $A$ and $C$ are left locally finite, then $({^{*\text{\rm -gr}}}C,{^{*\text{\rm -gr}}}\!A)$ is Koszul. Similarly, if $A$ and $C$ are right locally finite, then $(C^{*\text{\rm -gr}},A^{*\text{\rm -gr}})$ is Koszul.
\end{Theorem}

\begin{proof}
Let us assume that $(A,C)$ is a Koszul pair, so the complex $\operatorname{K}_\bullet^l({A},{C},m)$ is exact for all $m>1$, see Section~\ref{fa:Koszul pairs}. We will show that $\operatorname{K}^\bullet_l(\dg {C},\dg{A},m)$ is also exact, for all $m>0$, as a~consequence of the fact that it is isomorphic to the graded left dual of $\operatorname{K}_\bullet^l(A,C,m)$. In turn, the claimed isomorphism will be proved by showing that the diagram
\begin{gather*}
\xymatrix{0\! \ar[r] & \!{^*}\!A^m \!\ot\! {^*}C_0\! \ar[d]^-\phi_\sim \ar[r]^-{\partial^0} & {^*}\!A^{m-1} \!\ot\! {^*}C_1 \ar[d]^-\phi_\sim \ar[r]^-{\partial^1} & {\cdots} \ar[r]^-{\partial^{m-2}} & {^*}\!A^1 \!\ot\! {^*}C_{m-1}\! \ar[d]^-\phi_\sim \ar[r]^-{\partial^{m-1}} & {^*}\!A^0 \!\ot\! {^*}C_m \ar[d]^-\phi_\sim\! \ar[r] & \!0 \\
0\! \ar[r] & \!{^*}(A^m \!\ot\! C_0)\! \ar[r]^-{\de^0}&{^*} (A^{m-1} \!\ot\! C_1)\ar[r]^-{\de^1} & {\cdots} \ar[r]^-{\de^{m-2}} & {^*}(A^1 \!\ot\! C_{m-1})\! \ar[r]^-{\de^{m-1}} & {^*}(A^0 \!\ot\! C_m)\! \ar[r] & \!0 }
\end{gather*}
has commutative squares, where $\pa^{n}$ and $\de^n$ denote the dif\/ferential map of $\operatorname{K}^\bullet_l(\dg {C},\dg{A},m)$ and the transposed map of the restriction of $d_{m-n+1}^l$ to $A^{n-1}\ot C_{m-n+1}$, respectively.

Indeed, if $\theta:=\theta_{C,A}$ and we pick $\al\in{}^*\!A^n$, $\be\in{}^*C_{m-n}$, $a\in A^{n-1}$ and $c\in C_{m-n+1}$, then we have $\pa^n(\al\ot\be)=\sum \al_{1,n-1}\ot {}^*\theta(\al_{2,1})\ast\be$. Using the def\/inition of $\phi$, the def\/inition of the left convolution product, the fact $\theta$ is right $R$-linear and the equivalence between the relations \eqref{eq:sweedler} and \eqref{eq:Delta*-3}, we get
\begin{gather*}
\big((\phi\circ\pa^n)(\al\ot\be)\big)(a\ot c) =\sum\al_{1,n-1}\big(a ({}^*\theta(\al_{2,1})\ast\be\big)(c))\\
\hphantom{\big((\phi\circ\pa^n)(\al\ot\be)\big)(a\ot c)}{} =\sum\al_{1,n-1}\big(a\al_{2,1}\big(\theta(c_{1,1}\be(c_{2,m-n})\big)\big)\\
\hphantom{\big((\phi\circ\pa^n)(\al\ot\be)\big)(a\ot c)}{}
=\sum\al_{1,n-1}\big(a\al_{2,1}\big(\theta(c_{1,1})\be(c_{2,m-n})\big)\big)\\
\hphantom{\big((\phi\circ\pa^n)(\al\ot\be)\big)(a\ot c)}{}
=\sum\al\big(a\theta(c_{1,1})\be(c_{2,m-n})\big).
\end{gather*}
Similarly, since $\de^n$ is the transposed map of the restriction of $d^l_{m-n+1}$ to $A^{n-1}\ot C_{m-n+1}$, using once again the def\/inition of $\phi$, we have
\begin{gather*}
 \big((\de^n\circ\phi)(\al\ot\be)\big)(a\ot c) =\phi(\al\ot\be)\big(d^l_{m-n+1}(a\ot c)\big)=\phi(\al\ot\be)\Big(\sum a\theta( c_{1,1})\ot c_{2,m-n}\Big)\\
\hphantom{\big((\de^n\circ\phi)(\al\ot\be)\big)(a\ot c)}{}
 =\sum\al\big(a\theta(c_{1,1})\be(c_{2,m-n})\big).
\end{gather*}
By comparing the results of the computations from the two sequences of equations we deduce that the squares of the above diagram are commutative, as we claimed.

In a similar way one proves the second statement of the corollary.
\end{proof}

\begin{Corollary} \label{cor:aast}
Let $A$ be a connected graded $R$-ring. If $A$ is a left $($right$)$ locally finite Koszul $R$-ring then its left $($right$)$ graded dual is a Koszul $R^{\rm op}$-coring.
\end{Corollary}

\begin{proof}
Note that $A^!_n$ is a submodule of the f\/initely generated left $R$-module $A^1\ot \cdots \ot A^1$, where the tensor product has $n$ factors. Since $R$ is a Noetherian ring it follows that $A^!$ is left locally f\/inite. By Theorem~\ref{t-k7}, the pair $(A,A^!)$ is Koszul. Using the preceding theorem we deduce that $(^{*\text{-gr}}(A^!), ^{*\text{-gr}}\!A)$ is Koszul. In particular,
$^{*\text{-gr}}\!A$ is Koszul.
\end{proof}

\begin{Corollary} \label{cor:cast}
Let $C$ be a connected graded $R$-coring. If $C$ is a left $($right$)$ locally finite Koszul $R$-coring then its left $($right$)$ graded dual is a Koszul $R^{\rm op}$-ring.
\end{Corollary}

\begin{proof}
 One proceeds as in the proof of the preceding corollary.
\end{proof}

\begin{Corollary}
Let $A$ be a Koszul $R$-ring and let $C$ be a Koszul $R$-coring.
\begin{enumerate}\itemsep=0pt
 \item[$1.$] If $A$ and $C$ are left locally finite, then $E({}^{*\text{\rm -gr}}\!A)\cong {}^{*\text{\rm -gr}}T(A)$ and $T({}^{*\text{\rm -gr}}C)\cong {}^{*\text{\rm -gr}}E(C)$.

\item[$2.$] If $A$ and $C$ are right locally finite, then $E(A^{*\text{\rm -gr}})\cong T(A)^{*\text{\rm -gr}}$ and $T(C^{*\text{\rm -gr}})\cong E(C)^{*\text{\rm -gr}}$.
\end{enumerate}
\end{Corollary}

\begin{proof}
Let us assume that $A$ is left locally f\/inite. As $A$ is Koszul, then $(A,T(A))$ is a Koszul pair and $A^!\cong T(A)$, see Theorem \ref{t-k7}. By the proof of Corollary \ref{cor:aast}, the coring $A^!$ is left locally f\/inite. Hence, a fortiori, $T(A)$ is also left locally f\/inite and $(^{*\text{-gr}}T(A),{}^{*\text{-gr}}A)$ is Koszul. By \cite[Theorem~2.9]{jps}, for any Koszul pair $(B,D)$ the graded corings $D$ and $T(B)$ are isomorphic. In particular, for the Koszul pair $(^{*\text{-gr}}T(A),{}^{*\text{-gr}}A)$, we get $E({}^{*\text{-gr}}\!A)\cong {}^{*\text{-gr}}T(A)$.

The remaining three isomorphisms can be similarly proved.
\end{proof}

As an application of Corollaries~\ref{cor:aast} and~\ref{cor:cast} we are going to show that the incidence ring of a~f\/inite graded poset is Koszul if and only if the incidence coring of that poset is Koszul.

This class of examples was considered previously in~\cite{polo,wood}. They were also studied even in a~more general (nongraded) setting in~\cite{str}.

\section{Incidence (co)rings}\label{section6}
In this f\/inal section of the article, we will use the general results developed in the f\/irst three sections, restricting ourselves to the locally f\/inite context provided by the fourth section. The main direction of applications in this section is that of incidence algebras of f\/inite graded posets, which we endow with $R$-(co)ring structures, such that our theory of Koszul pairs and Koszul corings could be applied. Lastly, we add that these results could be seen as a framework in which particular cases of Koszul posets could be studied, as per \cite[Section~2]{mst}.

Since the theory is generally well established, we only state the particular hypotheses and notations which are used here. As such, we restrict our study to the case of \emph{finite graded posets}~$\kal{P}$, i.e., those in which every maximal chain included in an interval is of the same length (which we denote by $l([x,y])$ for the closed interval~$[x,y]$). Let~$\kal{I}_p$ denote the set of all intervals of length~$p$ in the poset.

Fixing a base f\/ield $\Bbbk$, denote by $\kap$ the \emph{incidence algebra of~$\kal{P}$}. As a vector space, it has a~basis of the form $\kal{B} = \{e_{x,y} \,|\, x \leq y \}$ and introducing a~multiplication
\begin{gather*}
	e_{x,y} \cdot e_{z,u} = \delta_{y,z}e_{x,u},
\end{gather*}
which is extended by linearity, we obtain the associative unital algebra structure. Note furthermore that $\Bbbk^a[\kal{P}]$ is an $R$-bimodule, where $R=\langle e_{x,x}\,|\, x\in\kal{P} \rangle \simeq \Bbbk^{|\kal{P}|}$ is regarded as a $\Bbbk$-algebra as above and the actions are def\/ined by the relation
\begin{gather*} e_{x,x}\cdot e_{y,z}\cdot e_{u,u}=\de_{x,y}\de_{z,u}e_{y,z}. \end{gather*}

Finally, $\kap$ becomes an $R$-ring, which will be referred to as \emph{the incidence $R$-ring of $\kal{P}$}.

By duality, to every f\/inite poset corresponds a coalgebra, namely its \emph{incidence coalgebra} $\Bbbk^c[\kal{P}]$, which as a linear space coincides with $\kap$. The comultiplication is given by the formula
\begin{gather*}
\Delta(e_{x,y}) = \sum_{z \in [x,y]} e_{x,z} \ot_\Bbbk e_{z,x}.\end{gather*}
The counit $\varepsilon$ is uniquely def\/ined such that $\varepsilon(e_{x,y})=\de_{x,y}$.

It is not dif\/f\/icult to see that $\xi\circ \Delta\colon \Bbbk^c[\kal{P}]\to \Bbbk^c[\kal{P}]\ot_R \Bbbk^c[\kal{P}]$ def\/ine an $R$-coring structure, where $\xi\colon \Bbbk^c[\kal{P}]\ot_\Bbbk \Bbbk^c[\kal{P}]\to \Bbbk^c[\kal{P}]\ot_R \Bbbk^c[\kal{P}]$ is the canonical map. The counit of this coring maps $e_{x,y}$ to $\de_{x,y}e_{x,x}$. The comultiplication and the counit of $\Bbbk^c[\kal{P}]$, regarded as a coring, will still be denoted by $\Delta$ and $\varepsilon$. Let us notice that
\begin{gather*}
\Delta(e_{x,y}) = \sum_{z \in [x,y]} e_{x,z} \ot e_{z,x},
\end{gather*}
where, as usual in this paper, $\ot=\ot_R$.

The relation between $\Bbbk^a[\kal{P}]$ and $\Bbbk^c[\kal{P}]$ is explained in the following result.

\begin{Theorem}\label{te:incidence-path}
We keep the notation from the preceding subsection. If $\kal{P}$ is a finite graded poset then $\Bbbk^a[\kal{P}]\cong {}^{*\text{\rm -gr}} \Bbbk^c[\kal{P}]$. In particular, the $R$-ring $\Bbbk^a[\kal{P}]$ is Koszul if and only if the $R$-coring $\Bbbk^c[\kal{P}] $ is Koszul.
\end{Theorem}

\begin{proof}
Let $A:=\Bbbk^a[\kal{P}]$ and $C:=\Bbbk^c[\kal{P}]$. Therefore, we have to prove that there exists an isomorphism of graded $R$-rings $A\cong {^{*\text{-gr}}}C$. For two comparable elements $x\leq y$ we def\/ine the $\Bbbk$-linear map $f_{x,y}\colon C_p\to R$ by $f_{x,y}(e_{u,v})=\de_{x,u}\de_{y,v}e_{u,u}$, for all $u\leq v$ such that $l([u,v])=p$. An easy computation shows that $f_{x,y}$ is left $R$-linear. Moreover, if $f\in {^*}C_p$ and $[x,y]\in\kal{I}_p$, then there is a a certain element $\al_{x,y}$ in $\Bbbk$ such that $f(e_{x,y})=\al_{x,y}e_{x,x}$. It follows that $f=\sum\limits_{[x,y]\in\kal{I}_p}\al_{x,y}f_{x,y}$. Thus $f$ can be written in a unique way as linear combination of the elements of $\{f_{x,y}\,|\, {[x,y]\in\kal{I}_p}\}$, so this set is a linear basis of ${^*}C_p$.

For $x\leq y$, $z\leq t$ and $v\leq w$ one proves, by a straightforward computation, that
\begin{gather*}
 (f_{x,y}\ast f_{z,t})(e_{v,w})= \begin{cases}
 e_{v,v}, &\text{if} \ y=z,\; x=v \ \text{and} \ t=w,\\
 0,&\text{otherwise}.
 \end{cases}
\end{gather*}
On the other hand, by def\/inition, $f_{x,t}(e_{v,w})=\de_{x,v}\de_{t,w}e_{v,v}$. Thus $f_{x,y}\ast f_{z,t}=\de_{y,z}f_{x,t}$, for all~$x$, $y$, $z$, $v$ and~$w$ as above.

Summarizing, the $\Bbbk$-linear map $\chi_p\colon A^p \to {}^*C_p$, def\/ined by $\chi_p(e_{x,y})=f_{x,y}$, for all $x\leq y$ such that $l([x,y])=p$, is the component of degree $p$ of an isomorphism of graded $R$-rings.

By Corollary \ref{cor:cast}, if $C$ is a Koszul coring, then $A\cong {}^{*\text{-gr}}C$ is a~Koszul ring. Both $A$ and $C$ are locally f\/inite, being f\/inite dimensional linear spaces. Thus, $C\cong {}(^{*\text{-gr}}C){}^{*\text{-gr}}\cong A^{*\text{-gr}}$. Hence $C$ is Koszul, provided that $A$ is so, cf.\ Corollary~\ref{cor:aast}.
\end{proof}
\begin{Definition}
 We say that a graded poset $\kal{P}$ is \textit{Koszul} if its incidence ring $\Bbbk^a[\kal{P} ]$ is Koszul.
\end{Definition}

For a graded poset $\kal{P}$ let us denote the homogeneous component of degree $1$ of its incidence ring by $V$. For every interval $[x,y]$ of length $2$ we def\/ine the element
\begin{gather*}
\zeta_{x,y}:=\sum_{z\in (x,y)}e_{x,z}\ot_R e_{z,y}.
\end{gather*}
Let $I_\kal{P}$ denote the ideal generated in $T^a_R(V)$ by the set $\{\zeta_{x,y}\,|\, l([x,y])=2 \}$. With this notation in our hands, we have the following result.
\begin{Theorem}\label{te:shrieck}
If $A$ is the incidence ring of a Koszul poset $\kal{P}$, then the $R$-ring $T_R^a(V)/I_\kal{P}$ is Koszul.
\end{Theorem}

\begin{proof}
By Theorem \ref{te:incidence-path} the incidence coring $C:=\Bbbk^c[\kal{P} ]$ is Koszul. Hence, in view of Theorem~\ref{thm:kcoring}, the $R$-ring $C^!$ is Koszul as well. We conclude by remarking that $C^!=T_R^a(V)/I_\kal{P}$, as $\zeta_{x,y}=\Delta_{1,1}(e_{x,y} ) $, for any interval $[x,y] $ of length~$2$.
\end{proof}

As f\/inal closing remarks, let us add that we have explored some rich examples of Koszul posets in \cite[Section~2]{mst}, where we also provided a constructive algorithm to produce new examples starting from old ones. All of the examples included in the cited article exhibit also Koszul corings, by Theorem \ref{te:incidence-path} above.

\subsection*{Acknowledgements}

Both authors were f\/inancially supported by CNCS-UEFISCDI, project PCE PN-II-ID-PCE-2011-3-0635.

\pdfbookmark[1]{References}{ref}
\LastPageEnding


\begin{thebibliography}{99}
\footnotesize\itemsep=0pt

\bibitem{bgs}
Beilinson A., Ginzburg V., Soergel W., Koszul duality patterns in
 representation theory, \href{http://dx.doi.org/10.1090/S0894-0347-96-00192-0}{\textit{J.~Amer. Math. Soc.}} \textbf{9} (1996),
 473--527.

\bibitem{ber}
Berger R., Koszulity for nonquadratic algebras, \href{http://dx.doi.org/10.1006/jabr.2000.8703}{\textit{J.~Algebra}}
 \textbf{239} (2001), 705--734.

\bibitem{brw}
Brzezinski T., Wisbauer R., Corings and comodules, \href{http://dx.doi.org/10.1017/CBO9780511546495}{\textit{London Mathematical
 Society Lecture Note Series}}, Vol.~309, Cambridge University Press,
 Cambridge, 2003.

\bibitem{grs}
Green E.L., Reiten I., Solberg {\O}., Dualities on generalized {K}oszul
 algebras, \href{http://dx.doi.org/10.1090/memo/0754}{\textit{Mem. Amer. Math. Soc.}} \textbf{159} (2002), xvi+67~pages.

\bibitem{jps}
Jara~Mart\'{\i}nez P., L\'opez Pe\~na J., \c{S}tefan D., Koszul pairs and
 applications, \textit{J.~Noncommut. Geom.}, {t}o appear, \href{http://arxiv.org/abs/1011.4243}{arXiv:1011.4243}.

\bibitem{ke}
Keller B., Koszul duality and coderived categories (after {K}.~{L}ef\`{e}vre),
 available at \url{https://webusers.imj-prg.fr/~bernhard.keller/publ/kdcabs.html}.

\bibitem{md1}
Madsen D.O., On a common generalization of {K}oszul duality and tilting
 equivalence, \href{http://dx.doi.org/10.1016/j.aim.2011.05.003}{\textit{Adv. Math.}} \textbf{227} (2011), 2327--2348,
 \href{http://arxiv.org/abs/1007.3282}{arXiv:1007.3282}.

\bibitem{md2}
Madsen D.O., Quasi-hereditary algebras and generalized {K}oszul duality,
 \href{http://dx.doi.org/10.1016/j.jalgebra.2013.08.005}{\textit{J.~Algebra}} \textbf{395} (2013), 96--110, \href{http://arxiv.org/abs/1201.0441}{arXiv:1201.0441}.

\bibitem{mst}
Manea A., \c{S}tefan D., On Koszulity of f\/inite graded posets,
 \textit{J.~Algebra Appl.}, {t}o appear, \href{http://arxiv.org/abs/1605.05458}{arXiv:1605.05458}.

\bibitem{may}
May J.P., Bialgebras and {H}opf algebras, available at
 \url{http://www.math.uchicago.edu/~may/TQFT/HopfAll.pdf}.

\bibitem{mos}
Mazorchuk V., Ovsienko S., Stroppel C., Quadratic duals, {K}oszul dual
 functors, and applications, \href{http://dx.doi.org/10.1090/S0002-9947-08-04539-X}{\textit{Trans. Amer. Math. Soc.}} \textbf{361}
 (2009), 1129--1172, \href{http://arxiv.org/abs/math.RT/0603475}{math.RT/0603475}.

\bibitem{pi}
Piontkovski D., Graded algebras and their dif\/ferentially graded extensions,
 \href{http://dx.doi.org/10.1007/s10958-007-0137-y}{\textit{J.~Math. Sci.}} \textbf{142} (2007), 2267--2301.

\bibitem{PP}
Polishchuk A., Positselski L., Quadratic algebras, \href{http://dx.doi.org/10.1090/ulect/037}{\textit{University Lecture
 Series}}, Vol.~37, Amer. Math. Soc., Providence, RI, 2005.

\bibitem{polo}
Polo P., On {C}ohen--{M}acaulay posets, {K}oszul algebras and certain modules
 associated to {S}chubert varieties, \href{http://dx.doi.org/10.1112/blms/27.5.425}{\textit{Bull. London Math. Soc.}}
 \textbf{27} (1995), 425--434.

\bibitem{Pr}
Priddy S.B., Koszul resolutions, \href{http://dx.doi.org/10.1090/S0002-9947-1970-0265437-8}{\textit{Trans. Amer. Math. Soc.}} \textbf{152}
 (1970), 39--60.

\bibitem{str}
Reiner V., Stamate D.I., Koszul incidence algebras, af\/f\/ine semigroups, and
 {S}tanley--{R}eisner ideals, \href{http://dx.doi.org/10.1016/j.aim.2010.02.005}{\textit{Adv. Math.}} \textbf{224} (2010),
 2312--2345, \href{http://arxiv.org/abs/0904.1683}{arXiv:0904.1683}.

\bibitem{wood}
Woodcock D., Cohen--{M}acaulay complexes and {K}oszul rings, \href{http://dx.doi.org/10.1112/S0024610798005717}{\textit{J.~London
 Math. Soc.}} \textbf{57} (1998), 398--410.

\end{thebibliography}
\end{document}